\documentclass[12pt]{amsart}
\usepackage{float,array,multirow}
\usepackage[pdftex]{xcolor}
\usepackage[utf8]{inputenc}
\usepackage[english]{babel}
\usepackage{tikz}
\usetikzlibrary{decorations.pathreplacing}
\usepackage{caption}
\usepackage{longtable}
\usepackage{amsmath}
\usepackage{amsthm}
\usepackage{bm}
\usetikzlibrary{calc,arrows}
\usepackage{pgfplots}
\usepackage{graphicx,subcaption}
\usepackage{float}
\usepackage{fullpage}
\usepackage{cite}
\usetikzlibrary{calc,arrows}

\newtheorem{theorem}{Theorem}[section]
\newtheorem{corollary}{Corollary}[theorem]
\newtheorem{lemma}{Lemma}[section]
\newtheorem{definition}{Definition}
\newtheorem{prop}{Proposition}[section]
\newtheorem{question}{Problem}

\newcommand\tr{(t,r)}
\newcommand\T{T_\infty}

\newdimen\R
\title{Broadcast Domination of Triangular Matchstick Graphs and the Triangular Lattice}

\author{Pamela E. Harris}
\address{Department of Mathematics and Statistics, Williams College, United States}
\email{peh2@williams.edu}
\thanks{P.\,E. Harris was supported by NSF award DMS-1620202.}

\author{Dalia K. Luque}
\address{Department of Mathematics and Statistics, Williams College, United States}
\email{dkl3@williams.edu}
\thanks{}

\author{Claudia Reyes Flores}
\address{Department of Mathematics and Statistics, Williams College, United States}
\email{cr8@williams.edu}
\thanks{}

\author{Nohemi Sepulveda}
\address{Department of Mathematics and Statistics, Williams College, United States}
\email{nms2@williams.edu}
\thanks{}

\keywords{Broadcast domination,
    Triangular grid,
    Triangular Matchstick Graph}
\date{\today}

\begin{document}

\maketitle
\begin{abstract}
Blessing, Insko, Johnson and Mauretour gave a generalization of the domination number of a graph $G=(V,E)$ called the $(t,r)$ broadcast domination number which depends on the positive integer parameters $t$ and $r$. In this setting, a vertex $v \in V$ is a broadcast vertex of transmission strength $t$ if it transmits a signal of strength $t-d(u,v)$ to every vertex $u \in V$, where $d(u,v)$ denotes the distance between vertices $u$ and $v$ and $d(u,v) <t$. Given a set of broadcast vertices $S\subseteq V$, the reception at vertex $u$ is the sum of the transmissions from the broadcast vertices in $S$. The set $S \subseteq V$ is called a $(t,r)$ broadcast dominating set if every vertex $u \in V$ has a reception strength $r(u) \geq r$ and for a finite graph $G$ the cardinality of a smallest broadcast dominating set is called the $(t,r)$ broadcast domination number of $G$. In this paper, we consider the infinite triangular grid graph and define efficient $(t,r)$ broadcast dominating sets as those broadcasts that minimize signal waste. Our main result constructs efficient $(t,r)$ broadcasts on the infinite triangular lattice for all $t\geq r\geq 1$. Using these broadcasts, we then provide upper bounds for the $(t,r)$ broadcast domination numbers for triangular matchstick graphs when $(t,r)\in\{(2,1),(3,1),(3,2),(4,1),(4,2),(4,3),(t,t)\}$.
\end{abstract}


\section{Introduction} \label{sec:introduction}

The graph domination problem considers a finite graph $G=(V,E)$ and determines the minimal cardinality of a subset of the vertices $D\subseteq V$, such that every vertex in $V$ is either in the set $D$ or adjacent to a vertex in $D$. A set of vertices satisfying these conditions is called a \emph{dominating set} and the cardinality of a smallest dominating set is called the \emph{domination number} of the graph and is denoted $\gamma(G)$. 
The subject of graph domination continues to receive much attention in the literature since its establishment approximately 40 years ago. Graph theorists have studied the domination of many families of graphs, including grid graphs, complete grid graphs, cross product of paths, and graphs with minimum degree ~\cite{cherifietal,completegridgraph,DomInGraphs,Chang,DomInGraphsWithMinDegree,Jacobson}. Moreover, many variations on graph domination exist, such as total domination, $r$-domination, rainbow domination, and exponential domination ~\cite{rdomination,rainbow,DominatingBroadcasts,Gordon}. 

In 2015, Blessing, Insko, Johnson and Mauretour gave a generalization of the domination number of a graph $G$ called the $\tr$ \emph{broadcast domination number}, which  depends on the nonnegative integral parameters $t$ and $r$ \cite{Blessing}. 
To give the broadcast domination number of a graph we first need the following definitions. 
We say that $v \in V$ is a \emph{broadcast vertex of transmission strength $t$} if it transmits a \emph{signal} of strength $t-d(u,v)$ to every vertex $u \in V$ with $d(u,v) <t$. Given a set of broadcast vertices $S\subseteq V$, we say that the \emph{reception} at vertex $u$ is the sum of the transmissions from the broadcast vertices in $S$. That is, the reception at vertex $u$ is $$r(u)= \sum\limits_{\substack{v \in S \\ u \in N_t(v)}}(t-d(u,v)),$$ where $S$ is a \textit{broadcasting set} {consisting} of broadcasting vertices of strength $t$ and {$N_{t}(v)=\{u\in V:d(u,v)<t\}$} is the \textit{broadcast neighborhood of $v$}.
The set $S \subseteq V$ is called a $(t,r)$ \emph{broadcast dominating set} if every vertex $u \in V$ has a reception strength $r(u) \geq r$ and {for a finite graph $G$} the cardinality of a smallest broadcast dominating set is called the $\tr$ broadcast domination number of $G$ and is denoted $\gamma_{t,r}(G)$. {In this setting, $(t,r)$ broadcast domination theory generalizes domination and distance
domination theories for graphs. We note that} the $(2,1)$ broadcast number is the classical domination number of a graph.

The initial work of Blessing et al. established the $(t,r)$ broadcast domination numbers for small grid graphs with $(t,r)\in\{(2,2),(3,1),(3,2),(3,3)\}$, and provided upper bounds for these broadcast domination numbers for arbitrarily large grid graphs \cite{Blessing}. 

In this paper, we consider the infinite triangular grid graph $\T$ and define \emph{efficient} $\tr$ broadcast dominating sets of~$\T$ as follows.

\begin{definition}\label{def:efficient}
A $(t,r)$ broadcast dominating set $S$ for $T_\infty$ is said to be  \emph{efficient} if 
\begin{align}
    r(u)=\begin{cases}r&\mbox{if $d(u,v)\geq t-r$ for all $v\in S$}\\t-d(u,v)&\mbox{if $0\leq d(u,v)< t-r$ for exactly one $v\in S$}.\end{cases}
\end{align}
\end{definition}

In other words, efficient $\tr$ broadcasts are those broadcasts on the infinite triangular grid graph that minimize wasted signal in the sense that every vertex far away from broadcast towers receive exactly the signal strength needed and no more, while those vertices close to a tower only get their signal from that tower and from no other.

\begin{figure}[H]
    \centering
    \resizebox{.3\textwidth}{!}{\begin{tikzpicture}
\clip (1.5,.5) rectangle (11,11);
\foreach \x  in {0,...,25} {
        \draw[gray] (\x-11,0) -- ($(\x-11,0)+20*(0.5, {0.5*sqrt(3)})$);
        \draw[gray] (\x,0) -- ($(\x,0)+20*(-0.5, {0.5*sqrt(3)})$);
        \draw[gray] ($\x*(0, {0.5*sqrt(3)})$) -- ($\x*(15, {0.5*sqrt(3)})$);
        \draw[gray] (\x-5,0) -- ($(\x-5,0)+20*(-0.5, {0.5*sqrt(3)})$);
        \draw[gray] ($\x*(0, {0.5*sqrt(3)})$) -- ($\x*(-1, {0.5*sqrt(3)})$);
    }
\begin{scope}[shift={(5.5,-0.86)}]
\draw[thick, black] (0:\R) \foreach \x in {60,120,...,359} {
                -- (\x:\R)
            }-- cycle (90:\R);
\fill [red] (0:0) circle (3pt);
\end{scope}    
    
\begin{scope}[shift={(2,1.73)}]
\draw[thick, black] (0:\R) \foreach \x in {60,120,...,359} {
                -- (\x:\R)
            }-- cycle (90:\R);
\fill [red] (0:0) circle (3pt);
\end{scope}

\begin{scope}[shift={(2.5,6.06)}]
\draw[thick, black] (0:\R) \foreach \x in {60,120,...,359} {
                -- (\x:\R)
            }-- cycle (90:\R);
\fill [red] (0:0) circle (3pt);
\end{scope}

\begin{scope}[shift={(3,10.39)}]
\draw[thick, black] (0:\R) \foreach \x in {60,120,...,359} {
                -- (\x:\R)
            }-- cycle (90:\R);
\fill [red] (0:0) circle (3pt);
\end{scope}

\begin{scope}[shift={(6,3.46)}]
\draw[thick, black] (0:\R) \foreach \x in {60,120,...,359} {
                -- (\x:\R)
            }-- cycle (90:\R);
\fill [red] (0:0) circle (3pt);
\end{scope}

\begin{scope}[shift={(6.5,7.79)}]
\draw[thick, black] (0:\R) \foreach \x in {60,120,...,359} {
                -- (\x:\R)
            }-- cycle (90:\R);
\fill [red] (0:0) circle (3pt);
\end{scope}

\begin{scope}[shift={(7,12.12)}]
\draw[thick, black] (0:\R) \foreach \x in {60,120,...,359} {
                -- (\x:\R)
            }-- cycle (90:\R);
\fill [red] (0:0) circle (3pt);
\end{scope}

\begin{scope}[shift={(9.5,0.86)}]
\draw[thick, black] (0:\R) \foreach \x in {60,120,...,359} {
                -- (\x:\R)
            }-- cycle (90:\R);
\fill [red] (0:0) circle (3pt);
\end{scope}

\begin{scope}[shift={(13.5,2.6)}]
\draw[thick, black] (0:\R) \foreach \x in {60,120,...,359} {
                -- (\x:\R)
            }-- cycle (90:\R);
\fill [red] (0:0) circle (3pt);
\end{scope}

\begin{scope}[shift={(14,6.92)}]
\draw[thick, black] (0:\R) \foreach \x in {60,120,...,359} {
                -- (\x:\R)
            }-- cycle (90:\R);
\fill [red] (0:0) circle (3pt);
\end{scope}

\begin{scope}[shift={(10,5.19)}]
\draw[thick, black] (0:\R) \foreach \x in {60,120,...,359} {
                -- (\x:\R)
            }-- cycle (90:\R);
\fill [red] (0:0) circle (3pt);
\end{scope}

\begin{scope}[shift={(10.5,9.52)}]
\draw[thick, black] (0:\R) \foreach \x in {60,120,...,359} {
                -- (\x:\R)
            }-- cycle (90:\R);
\fill [red] (0:0) circle (3pt);
\end{scope}

\end{tikzpicture}}\hfill
    \resizebox{.3\textwidth}{!}{\begin{tikzpicture}
\clip (2,1) rectangle (11,11);
\foreach \x  in {0,...,25} {
        \draw[gray] (\x-11,0) -- ($(\x-11,0)+20*(0.5, {0.5*sqrt(3)})$);
        \draw[gray] (\x,0) -- ($(\x,0)+20*(-0.5, {0.5*sqrt(3)})$);
        \draw[gray] ($\x*(0, {0.5*sqrt(3)})$) -- ($\x*(15, {0.5*sqrt(3)})$);
        \draw[gray] (\x-5,0) -- ($(\x-5,0)+20*(-0.5, {0.5*sqrt(3)})$);
        \draw[gray] ($\x*(0, {0.5*sqrt(3)})$) -- ($\x*(-1, {0.5*sqrt(3)})$);
    }
    
\begin{scope}[shift={(2,1.75)}]
\draw[thick, black] (0:\R) \foreach \x in {60,120,...,359} {
                -- (\x:\R)
            }-- cycle (90:\R);
\fill [red] (0:0) circle (3pt);
\end{scope}

\begin{scope}[shift={(3,5.2)}]
\draw[thick, black] (0:\R) \foreach \x in {60,120,...,359} {
                -- (\x:\R)
            }-- cycle (90:\R);
\fill [red] (0:0) circle (3pt);
\end{scope}

\begin{scope}[shift={(3,5.2)}]
\draw[thick, black] (0:\R) \foreach \x in {60,120,...,359} {
                -- (\x:\R)
            }-- cycle (90:\R);
\fill [red] (0:0) circle (3pt);
\end{scope}

\begin{scope}[shift={(4,8.65)}]
\draw[thick, black] (0:\R) \foreach \x in {60,120,...,359} {
                -- (\x:\R)
            }-- cycle (90:\R);
\fill [red] (0:0) circle (3pt);
\end{scope}

\begin{scope}[shift={(7.5,9.53)}]
\draw[thick, black] (0:\R) \foreach \x in {60,120,...,359} {
                -- (\x:\R)
            }-- cycle (90:\R);
\fill [red] (0:0) circle (3pt);
\end{scope}

\begin{scope}[shift={(6.5,6.06)}]
\draw[thick, black] (0:\R) \foreach \x in {60,120,...,359} {
                -- (\x:\R)
            }-- cycle (90:\R);
\fill [red] (0:0) circle (3pt);
\end{scope}

\begin{scope}[shift={(5.5,2.58)}]
\draw[thick, black] (0:\R) \foreach \x in {60,120,...,359} {
                -- (\x:\R)
            }-- cycle (90:\R);
\fill [red] (0:0) circle (3pt);
\end{scope}

\begin{scope}[shift={(8,0)}]
\draw[thick, black] (0:\R) \foreach \x in {60,120,...,359} {
                -- (\x:\R)
            }-- cycle (90:\R);
\fill [red] (0:0) circle (3pt);
\end{scope}

\begin{scope}[shift={(9,3.46)}]
\draw[thick, black] (0:\R) \foreach \x in {60,120,...,359} {
                -- (\x:\R)
            }-- cycle (90:\R);
\fill [red] (0:0) circle (3pt);
\end{scope}

\begin{scope}[shift={(10,6.95)}]
\draw[thick, black] (0:\R) \foreach \x in {60,120,...,359} {
                -- (\x:\R)
            }-- cycle (90:\R);
\fill [red] (0:0) circle (3pt);
\end{scope}

\begin{scope}[shift={(11,10.42)}]
\draw[thick, black] (0:\R) \foreach \x in {60,120,...,359} {
                -- (\x:\R)
            }-- cycle (90:\R);
\fill [red] (0:0) circle (3pt);
\end{scope}

\begin{scope}[shift={(.5,7.8)}]
\draw[thick, black] (0:\R) \foreach \x in {60,120,...,359} {
                -- (\x:\R)
            }-- cycle (90:\R);
\fill [red] (0:0) circle (3pt);
\end{scope}

\begin{scope}[shift={(1.5,11.26)}]
\draw[thick, black] (0:\R) \foreach \x in {60,120,...,359} {
                -- (\x:\R)
            }-- cycle (90:\R);
\fill [red] (0:0) circle (3pt);
\end{scope}

\begin{scope}[shift={(5,12.12)}]
\draw[thick, black] (0:\R) \foreach \x in {60,120,...,359} {
                -- (\x:\R)
            }-- cycle (90:\R);
\fill [red] (0:0) circle (3pt);
\end{scope}

\begin{scope}[shift={(11.5,0.86)}]
\draw[thick, black] (0:\R) \foreach \x in {60,120,...,359} {
                -- (\x:\R)
            }-- cycle (90:\R);
\fill [red] (0:0) circle (3pt);
\end{scope}

\begin{scope}[shift={(12.5,4.33)}]
\draw[thick, black] (0:\R) \foreach \x in {60,120,...,359} {
                -- (\x:\R)
            }-- cycle (90:\R);
\fill [red] (0:0) circle (3pt);
\end{scope}

\end{tikzpicture}}\hfill
    \resizebox{.3\textwidth}{!}{\begin{tikzpicture}
\clip (2,1) rectangle (11,11);
\foreach \x  in {0,...,25} {
        \draw[gray] (\x-11,0) -- ($(\x-11,0)+20*(0.5, {0.5*sqrt(3)})$);
        \draw[gray] (\x,0) -- ($(\x,0)+20*(-0.5, {0.5*sqrt(3)})$);
        \draw[gray] ($\x*(0, {0.5*sqrt(3)})$) -- ($\x*(15, {0.5*sqrt(3)})$);
        \draw[gray] (\x-5,0) -- ($(\x-5,0)+20*(-0.5, {0.5*sqrt(3)})$);
        \draw[gray] ($\x*(0, {0.5*sqrt(3)})$) -- ($\x*(-1, {0.5*sqrt(3)})$);
    }
    
\begin{scope}[shift={(2,1.73)}]
\draw[thick, black] (0:\R) \foreach \x in {60,120,...,359} {
                -- (\x:\R)
            }-- cycle (90:\R);
\fill [red] (0:0) circle (3pt);
\end{scope}

\begin{scope}[shift={(3.5,4.33)}]
\draw[thick, black] (0:\R) \foreach \x in {60,120,...,359} {
                -- (\x:\R)
            }-- cycle (90:\R);
\fill [red] (0:0) circle (3pt);
\end{scope}

\begin{scope}[shift={(5,1.73)}]
\draw[thick, black] (0:\R) \foreach \x in {60,120,...,359} {
                -- (\x:\R)
            }-- cycle (90:\R);
\fill [red] (0:0) circle (3pt);
\end{scope}

\begin{scope}[shift={(6.5,4.33)}]
\draw[thick, black] (0:\R) \foreach \x in {60,120,...,359} {
                -- (\x:\R)
            }-- cycle (90:\R);
\fill [red] (0:0) circle (3pt);
\end{scope}

\begin{scope}[shift={(8,1.73)}]
\draw[thick, black] (0:\R) \foreach \x in {60,120,...,359} {
                -- (\x:\R)
            }-- cycle (90:\R);
\fill [red] (0:0) circle (3pt);
\end{scope}

\begin{scope}[shift={(9.5,4.33)}]
\draw[thick, black] (0:\R) \foreach \x in {60,120,...,359} {
                -- (\x:\R)
            }-- cycle (90:\R);
\fill [red] (0:0) circle (3pt);
\end{scope}

\begin{scope}[shift={(11,1.73)}]
\draw[thick, black] (0:\R) \foreach \x in {60,120,...,359} {
                -- (\x:\R)
            }-- cycle (90:\R);
\fill [red] (0:0) circle (3pt);
\end{scope}

\begin{scope}[shift={(2,6.93)}]
\draw[thick, black] (0:\R) \foreach \x in {60,120,...,359} {
                -- (\x:\R)
            }-- cycle (90:\R);
\fill [red] (0:0) circle (3pt);
\end{scope}

\begin{scope}[shift={(3.5,9.52)}]
\draw[thick, black] (0:\R) \foreach \x in {60,120,...,359} {
                -- (\x:\R)
            }-- cycle (90:\R);
\fill [red] (0:0) circle (3pt);
\end{scope}

\begin{scope}[shift={(5,6.93)}]
\draw[thick, black] (0:\R) \foreach \x in {60,120,...,359} {
                -- (\x:\R)
            }-- cycle (90:\R);
\fill [red] (0:0) circle (3pt);
\end{scope}

\begin{scope}[shift={(6.5,9.53)}]
\draw[thick, black] (0:\R) \foreach \x in {60,120,...,359} {
                -- (\x:\R)
            }-- cycle (90:\R);
\fill [red] (0:0) circle (3pt);
\end{scope}

\begin{scope}[shift={(8,6.93)}]
\draw[thick, black] (0:\R) \foreach \x in {60,120,...,359} {
                -- (\x:\R)
            }-- cycle (90:\R);
\fill [red] (0:0) circle (3pt);
\end{scope}

\begin{scope}[shift={(9.5,9.53)}]
\draw[thick, black] (0:\R) \foreach \x in {60,120,...,359} {
                -- (\x:\R)
            }-- cycle (90:\R);
\fill [red] (0:0) circle (3pt);
\end{scope}

\begin{scope}[shift={(12.5,9.53)}]
\draw[thick, black] (0:\R) \foreach \x in {60,120,...,359} {
                -- (\x:\R)
            }-- cycle (90:\R);
\fill [red] (0:0) circle (3pt);
\end{scope}

\begin{scope}[shift={(.5,9.53)}]
\draw[thick, black] (0:\R) \foreach \x in {60,120,...,359} {
                -- (\x:\R)
            }-- cycle (90:\R);
\fill [red] (0:0) circle (3pt);
\end{scope}

\begin{scope}[shift={(11,6.935)}]
\draw[thick, black] (0:\R) \foreach \x in {60,120,...,359} {
                -- (\x:\R)
            }-- cycle (90:\R);
\fill [red] (0:0) circle (3pt);
\end{scope}

\begin{scope}[shift={(11,12.13)}]
\draw[thick, black] (0:\R) \foreach \x in {60,120,...,359} {
                -- (\x:\R)
            }-- cycle (90:\R);
\fill [red] (0:0) circle (3pt);
\end{scope}

\begin{scope}[shift={(8,12.13)}]
\draw[thick, black] (0:\R) \foreach \x in {60,120,...,359} {
                -- (\x:\R)
            }-- cycle (90:\R);
\fill [red] (0:0) circle (3pt);
\end{scope}

\begin{scope}[shift={(5,12.13)}]
\draw[thick, black] (0:\R) \foreach \x in {60,120,...,359} {
                -- (\x:\R)
            }-- cycle (90:\R);
\fill [red] (0:0) circle (3pt);
\end{scope}

\begin{scope}[shift={(2,12.13)}]
\draw[thick, black] (0:\R) \foreach \x in {60,120,...,359} {
                -- (\x:\R)
            }-- cycle (90:\R);
\fill [red] (0:0) circle (3pt);
\end{scope}

\begin{scope}[shift={(12.5,4.33)}]
\draw[thick, black] (0:\R) \foreach \x in {60,120,...,359} {
                -- (\x:\R)
            }-- cycle (90:\R);
\fill [red] (0:0) circle (3pt);
\end{scope}

\begin{scope}[shift={(.5,4.33)}]
\draw[thick, black] (0:\R) \foreach \x in {60,120,...,359} {
                -- (\x:\R)
            }-- cycle (90:\R);
\fill [red] (0:0) circle (3pt);
\end{scope}

\end{tikzpicture}}
\caption{Efficient $(3,1)$, $(3,2)$, and $(3,3)$ broadcasts on $\T$.}
    \label{fig:exampleofefficient}
\end{figure}
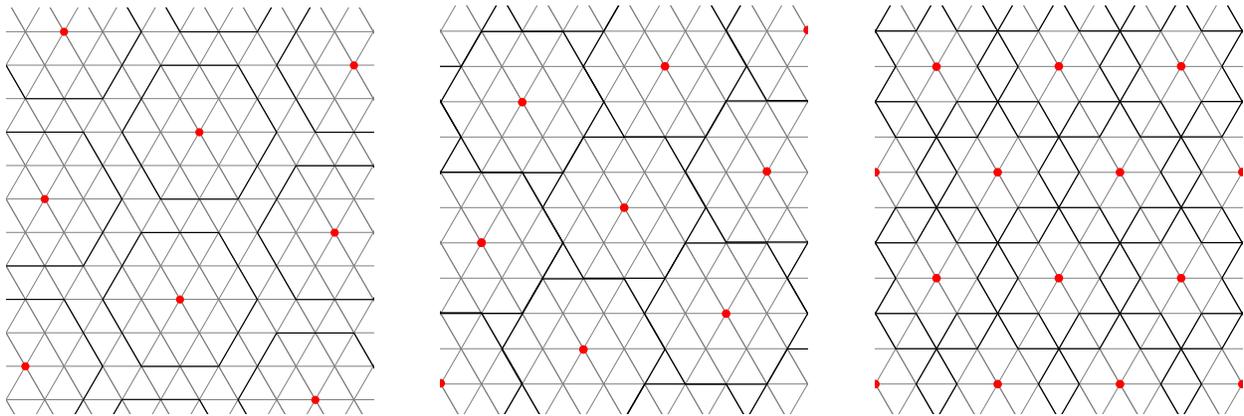

Figure \ref{fig:exampleofefficient} provides efficient $(3,1)$, $(3,2)$, and $(3,3)$ broadcasts on $\T$. In this figure, the red vertices are broadcasting vertices (which we also refer to as towers or dominators) while the vertices lying on the black boundaries are receiving signal 1 from each of the nearby towers. Note that the only ``waste'' present is innate waste, in the sense that each tower provides itself more signal than it needs since the signal strength is larger than the needed reception, but all remaining vertices receive exactly the signal required.

Note that Definition \ref{def:efficient} holds for any graph. However, we state it in terms of $T_\infty$, the graph we consider and for which we construct efficient $\tr$ broadcasts for all $t\geq r\geq 1$. 
We remark that efficient $\tr$ broadcasts of a graph are rare. For example, Drews, Harris, and Randolph recently considered the integer lattice graph and found optimal densities for $\tr$ broadcasts for all $t\geq 2$  and $r=1,2$. In their work density was defined intuitively as the proportion of the vertices of the infinite lattice graph contained in an infinite broadcast, i.e. the set of broadcasting vertices \cite{DrewsHarrisRandolph}. While it is currently unknown if there exist other efficient/optimal $\tr$ broadcasts for the integer lattice graph, our main result establishes efficient dominating patterns for the infinite triangular grid for all $t\geq r\geq 1$.
Using the results for the infinite triangular grid graph, we compute the $\tr$ broadcast domination number for triangular matchstick graphs $T_n$ for  some small $n$ and establish lower and upper bounds for the $(t,r)$ broadcast domination number for arbitrarily large triangular matchstick graphs  and $t \geq r\geq 1$.

This paper is organized as follows. Section \ref{sec:background} provides the necessary definitions and technical results to make our approach precise.
Section \ref{sec:finiteGn} introduces the concept of the reach of a dominator. We use this to compute the maximum size of a triangular matchstick graph so that the $(t,1)$ broadcast domination number is at most 3 and provide an upper bound for the $(t,1)$ domination number of $T_n$ for all $t,n\geq 1$.
Section \ref{sec:infinitegrid} studies the infinite triangular grid $T_\infty$ and presents efficient $(t,r)$ broadcast dominating patterns for all $t\geq r\geq 1$.
Using these broadcast dominating patterns, in Section \ref{sec:newbounds} we construct upper bounds for the $(t,r)$ domination number of $T_n$ for all $n\geq 1$ with $(t,r)\in\{(2,1),(3,1),(3,2),(4,1),(4,2),(4,3),(t,t)\}$. 
We end the paper with a section on open problems with future directions for research.

\section{Background}\label{sec:background}

In this section we set notation and provide the necessary definitions to make our approach precise.
We begin by defining the \textit{infinite triangular grid graph}, denoted $\T$, as the graph whose vertex set, $V(\T)$, can be put in a correspondence with points $(x,y) = \left (\frac{1}{2}a - b, \frac{\sqrt{2}}{3}a \right)$, where $a,b \in \mathbb{Z}$, and where two vertices are adjacent only if their corresponding coordinate entries are a Eucledian unit distance apart. We let $T_n$ denote the finite triangular grid graph, also known as the \textit{triangular matchstick arrangement graph of side $n$}, which is a subgraph of $\T$ with vertex set
 $\left\{ \left( \frac{1}{2}a - b, \frac{\sqrt{3}}{2}a \right) \;:\; a,b \in \mathbb{Z}, 0 \leq b \leq a \leq n \right\}. $
Figure \ref{fig:interioredges} illustrates the triangular matchstick graphs $T_n$ for $n=1,2,3,4$.
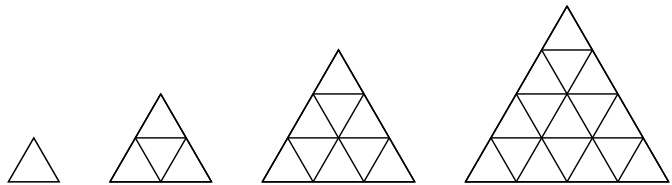
\begin{figure}[h]
    \centering
    \resizebox{3.5in}{!}{%
    \begin{tikzpicture}
\draw[line width=0.3mm, black](0,0)--(.5,.866)--(1,0)--cycle;

\begin{scope}[shift={(2,0)}]
    \foreach \row in {0, 1, ..., 2} {
        \draw[thick,black] ($\row*(0.5, {0.5*sqrt(3)})$) -- ($(2,0)+\row*(-0.5, {0.5*sqrt(3)})$);
        \draw[thick,black] ($\row*(1, 0)$) -- ($(2/2,{2/2*sqrt(3)})+\row*(0.5,{-0.5*sqrt(3)})$);
        \draw[thick,black] ($\row*(1, 0)$) -- ($(0,0)+\row*(0.5,{0.5*sqrt(3)})$);
    }
\draw[line width=0.3mm, black](0,0)--(1,1.73)--(2,0)--cycle;
\end{scope}

\begin{scope}[shift={(5,0)}]
    \foreach \row in {0, 1, ..., 3} {
        \draw[thick,black] ($\row*(0.5, {0.5*sqrt(3)})$) -- ($(3,0)+\row*(-0.5, {0.5*sqrt(3)})$);
        \draw[thick,black] ($\row*(1, 0)$) -- ($(3/2,{3/2*sqrt(3)})+\row*(0.5,{-0.5*sqrt(3)})$);
        \draw[thick,black] ($\row*(1, 0)$) -- ($(0,0)+\row*(0.5,{0.5*sqrt(3)})$);
    }
\draw[line width=0.3mm, black](0,0)--(1.5,2.6)--(3,0)--cycle;
\end{scope}

\begin{scope}[shift={(9,0)}]
    \foreach \row in {0,1,...,4} {
        \draw[thick,black] ($\row*(0.5, {0.5*sqrt(3)})$) -- ($(4,0)+\row*(-0.5, {0.5*sqrt(3)})$);
        \draw[thick,black] ($\row*(1, 0)$) -- ($(4/2,{4/2*sqrt(3)})+\row*(0.5,{-0.5*sqrt(3)})$);
        \draw[thick,black] ($\row*(1, 0)$) -- ($(0,0)+\row*(0.5,{0.5*sqrt(3)})$);
    }
\draw[line width=0.3mm, black](0,0)--(2,3.46)--(4,0)--cycle;
\end{scope}
\end{tikzpicture}        }
    \caption{Matchstick graphs $T_1$, $T_2$, $T_3$, and $T_4$.}
    \label{fig:interioredges}
\end{figure}

In determining the optimal placement of the dominating vertices on $T_{n}$ we first show that each dominating vertex of strength $t$ has a certain \textit{reach}, denoted by $\rho(t)$. The reach of a dominator is defined as the maximum area (in triangles) on the infinite triangular grid graph that receives signal from that dominating vertex. Figure \ref{reach for (3,1)} illustrates the reach of a broadcast vertex of transmission strength $t=3$ (marked by the bold vertex) along with the reception for each vertex (marked in red) within its reach. This shows that for $t=3$ the reach of a dominator is $\rho(3) = 24$. 

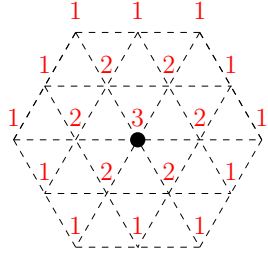
\begin{figure}[H]
    \centering
    \resizebox{1.5in}{!}{%
\begin{tikzpicture}
\coordinate (0;0) at (0,0); 
\foreach \c in {1,...,2}{%
\foreach \i in {0,...,5}{%
\pgfmathtruncatemacro\j{\c*\i}
\coordinate (\c;\j) at (60*\i:\c);  
} }
\foreach \i in {0,2,...,10}{%
\pgfmathtruncatemacro\j{mod(\i+2,12)} 
\pgfmathtruncatemacro\k{\i+1}
\coordinate (2;\k) at ($(2;\i)!.5!(2;\j)$) ;}

 \foreach \i in {0,...,6}{%
 \pgfmathtruncatemacro\k{\i}
 \pgfmathtruncatemacro\l{10-\i}
 \draw[dashed] (2;\k)--(2;\l); 
 \pgfmathtruncatemacro\k{6-\i} 
 \pgfmathtruncatemacro\l{mod(8+\i,12)}   
 \draw[dashed] (2;\k)--(2;\l); 
 \pgfmathtruncatemacro\k{9-\i} 
 \pgfmathtruncatemacro\l{mod(9+\i,12)}   
 \draw[dashed] (2;\k)--(2;\l);} 
 
\fill [black] (0;0) circle (3.5pt) node[above=2pt, red] {3};

 \foreach \n in {0,1,...,5}{%
   \draw (1;\n) circle (0pt) node[above=2pt, red] {2};}
 \foreach \n in {0,1,...,11}{%
   \draw (2;\n) circle (0pt) node[above=2pt, red] {1};}

\end{tikzpicture}}
    \caption{The reach of a dominator of strength $t=3$.}
    \label{reach for (3,1)}
\end{figure}

We begin by counting the number of vertices in $T_n$ and the number of interior edges in $T_n$ (the edges not forming the outer boundary), as these results are used in later sections. 

\begin{lemma}\label{lem:verticesinTn}
Let $n\geq 1$ and $\Delta_{n}=\binom{n+1}{2}$ denote the $n$-th triangular number. Then $\Delta_{n}$ counts the number of vertices of the triangular matchstick graph $T_{n-1}$. 
\end{lemma}
The proof of Lemma \ref{lem:verticesinTn} follows from the fact that the number of vertices in $T_{n-1}$ is equal to $\sum_{i=1}^{n}i=\binom{n+1}{2}$.

\begin{lemma}\label{lem:interioredgesinTn}
If $n\geq 1$, then the number of interior edges in $T_n$ is  $3\Delta_{n-1}$.
\end{lemma}
\begin{proof}
By Figure \ref{fig:interioredges}, we can see that the number of interior edges in the graphs $T_1$, $T_2$, $T_3$ and $T_4$ are $0,3,9,18$, respectively, which are equal to $3\Delta_0$, $3\Delta_1$, 
$3\Delta_2$ and $3\Delta_3$, respectively. Assume by induction that the number of interior edges in $T_n$ is given by $3\Delta_{n-1}$. 
Now notice that in creating $T_{n+1}$ we can add an additional row of triangles to the base of $T_n$. 
This would then turn the $n$ edges that were the base of $T_n$ into interior edges in $T_{n+1}$. 
Also for every edge on the base of $T_{n+1}$ we must have 2 edges with which we create a triangle ($T_1$). 
With the exception of the 2 outer most edges all of these edges are interior edges in $T_{n+1}$. Thus the total number of interior edges in $T_{n+1}$ is given by $3\Delta_{n-1}+n+(2(n+1)-2)=3\Delta_{n}$.
\end{proof}

We now show that the reach of every dominator depends only on the parameter $t$. 

\begin{lemma}\label{lemma:reach for general t}
For $t\geq1$, the \textit{reach} of dominator of strength $t$ is $\rho(t) = 6(t-1)^2$ and contains $3t^2-3t+1$ vertices.
\end{lemma}
\begin{proof}
First observe that $\rho(t)$ is counting the number of triangles within the reach of a dominator with strength $t$. The geometric shape of the reach of a dominator on a triangular grid is hexagonal in shape, and is made up of six triangular matchstick graphs with side length $t-1$. Since each $T_{t-1}$ covers exactly $(t-1)^2$ triangles on the triangular grid and since  six $T_{t-1}$ make up the reach of a strength $t$ dominator, we have shown that $\rho(t)=6(t-1)^2$.

To count the number of vertices in the reach of a dominator of strength $t$ it suffices to note that each $T_{t-1}$ has $\binom{t+1}{2}$ vertices, hence six such graphs contain $6\binom{t+1}{2}$ vertices. However, when we put these matchstick graphs together to create the hexagonal region defining the reach of the dominator, we are double counting the vertices along the 2 edges of every matchstick graph, which contain $t$ vertices. Thus, we must substract $6t$ and we must add 1 so that we account for the center vertex, since this subtraction eliminates that vertex from the count. Therefore, the number of vertices in the reach of a dominator of strength $t$ is given by $6\binom{t+1}{2}-6t+1=3t^2-3t+1.$ 
\end{proof}

\section{Triangular matchstick graphs}\label{sec:finiteGn} 

In this section we study the $(t,r)$ domination numbers of triangular matchstick graphs~$T_n$. Our first result provides a way to bound these domination numbers by varying the parameters $t$, $r$, and $n$.

\begin{prop}\label{prop:bounding}
If $t>r\geq 1$ and  $n\in\mathbb{N}$, then \begin{enumerate}
    \item $\gamma_{t,r}(T_n)\leq \gamma_{t-1,r}(T_n)$,
    \item $\gamma_{t,r}(T_n)\leq \gamma_{t,r+1}(T_n)$, and 
    \item $\gamma_{t,r}(T_{n})\leq \gamma_{t,r}(T_{n+1})$.
\end{enumerate}
\end{prop}

\begin{proof}
Statement (1) follows from the fact that if the strength of the signal of a dominator is lower, but the needed reception is fixed, then a graph may possibly require more dominators under a smaller signal strength. Statement (2) follows from noting that increasing the reception needed may possibly require more dominators. Statement (3) follows from the fact that a larger triangular graph may possibly need more dominators than a smaller one.
\end{proof}

We now specialize to $r=1$ and begin by providing a lower bound for the $(t,1)$ broadcast domination numbers for $T_n$.

\begin{lemma}\label{lemma:lowerbound} If $t \geq 2$ and $n\geq 1$, then
$ \gamma_{t,1}(T_n)\geq \left \lceil\frac{\Delta_{n+1}}{(3t^2-3t+1)}\right\rceil.$
\end{lemma}
\begin{proof}
In order to dominate $T_n$ we need to be able to overlay the reach of the dominators onto $T_n$ such that all vertices receive the minimum needed reception, which in this case is 1. We construct this lower bound by counting the vertices of $T_n$, which is given by $\Delta_{n+1}$,  and then dividing by the number of vertices in the reach of a dominator with transmission strength $t$. Taking the ceiling of this number, as we cannot use a fractional number of dominators, completes the result.
\end{proof}

We now compute some small $(t,1)$ broadcast domination numbers for $T_n$.

\begin{lemma}\label{lem:smalln t=1}
If $1\leq n\leq \lfloor\frac{3(t-1)}{2}\rfloor$, then $\gamma_{t,1}(T_n)=1$. 
\end{lemma}
\begin{proof}
It suffices to show that $\lfloor\frac{3(t-1)}{2}\rfloor$ is the side length of the largest equilateral triangle fitting within a regular hexagon with side length $t-1$. Since we orient triangular graphs upwards, in order to maximize the size of the triangle we note that its top corner should be placed closest to the midpoint of the top side of the hexagon created by the reach of a dominator. Placing the top of the triangle in this way divides the top side of the hexagon into a segment of length $\lceil\frac{t-1}{2}\rceil$ and another of length $t-1-\lceil\frac{t-1}{2}\rceil$. Thus the maximum size of a triangle fitting within the hexagon is given by subtracting $\lceil\frac{t-1}{2}\rceil$ from the diameter of the hexagon, which equals $2(t-1)-\lceil\frac{t-1}{2}\rceil=\lfloor\frac{3(t-1)}{2}\rfloor$, as claimed.
\end{proof}

\begin{figure}[H]
\centering
\resizebox{5in}{!}{%
\begin{tabular}[c]{cc}
   \begin{subfigure}{0.3\textwidth}
   \centering
   \resizebox{1.5in}{!}{
\begin{tikzpicture}
\begin{scope}[shift={(-.5,-.86)}]
    \foreach \row in {0,1} {
        \draw[thick] ($\row*(0.5, {0.5*sqrt(3)})$) -- ($(1,0)+\row*(-0.5, {0.5*sqrt(3)})$);
        \draw[thick] ($\row*(1, 0)$) -- ($(1/2,{1/2*sqrt(3)})+\row*(0.5,{-0.5*sqrt(3)})$);
        \draw[thick] ($\row*(1, 0)$) -- ($(0,0)+\row*(0.5,{0.5*sqrt(3)})$);
    }
\end{scope}
\coordinate (0;0) at (0,0); 
\foreach \c in {1,...,2}{%
\foreach \i in {0,...,5}{%
\pgfmathtruncatemacro\j{\c*\i}
\coordinate (\c;\j) at (60*\i:\c);  
} }
\foreach \i in {0,2,...,10}{%
\pgfmathtruncatemacro\j{mod(\i+2,12)}
\pgfmathtruncatemacro\k{\i+1}
\coordinate (2;\k) at ($(2;\i)!.5!(2;\j)$) ;}

 \foreach \i in {0,...,6}{%
 \pgfmathtruncatemacro\k{\i}
 \pgfmathtruncatemacro\l{10-\i}
 \draw[dashed] (2;\k)--(2;\l); 
 \pgfmathtruncatemacro\k{6-\i} 
 \pgfmathtruncatemacro\l{mod(8+\i,12)}   
 \draw[dashed] (2;\k)--(2;\l); 
 \pgfmathtruncatemacro\k{9-\i} 
 \pgfmathtruncatemacro\l{mod(9+\i,12)}   
 \draw[dashed] (2;\k)--(2;\l);} 
 
\fill [black] (0;0) circle (3.5pt) node[above=2pt, red] {3};

 \foreach \n in {0,1,...,5}{%
   \draw (1;\n) circle (0pt) node[above=2pt, red] {2};}
 \foreach \n in {0,1,...,11}{%
   \draw (2;\n) circle (0pt) node[above=2pt, red] {1};}
\end{tikzpicture}    }
\caption{}\label{fig:my_label2}
    \end{subfigure}

    \begin{subfigure}{0.3\textwidth}
    \centering
    \resizebox{1.5in}{!}{
\begin{tikzpicture}
\begin{scope}[shift={(-1.5,-.86)}]
    \foreach \row in {0,1,...,2} {
        \draw[thick] ($\row*(0.5, {0.5*sqrt(3)})$) -- ($(2,0)+\row*(-0.5, {0.5*sqrt(3)})$);
        \draw[thick] ($\row*(1, 0)$) -- ($(2/2,{2/2*sqrt(3)})+\row*(0.5,{-0.5*sqrt(3)})$);
        \draw[thick] ($\row*(1, 0)$) -- ($(0,0)+\row*(0.5,{0.5*sqrt(3)})$);
    }
\end{scope}
\coordinate (0;0) at (0,0); 
\foreach \c in {1,...,2}{%
\foreach \i in {0,...,5}{%
\pgfmathtruncatemacro\j{\c*\i}
\coordinate (\c;\j) at (60*\i:\c);  
} }
\foreach \i in {0,2,...,10}{%
\pgfmathtruncatemacro\j{mod(\i+2,12)}
\pgfmathtruncatemacro\k{\i+1}
\coordinate (2;\k) at ($(2;\i)!.5!(2;\j)$) ;}

 \foreach \i in {0,...,6}{%
 \pgfmathtruncatemacro\k{\i}
 \pgfmathtruncatemacro\l{10-\i}
 \draw[dashed] (2;\k)--(2;\l); 
 \pgfmathtruncatemacro\k{6-\i} 
 \pgfmathtruncatemacro\l{mod(8+\i,12)}   
 \draw[dashed] (2;\k)--(2;\l); 
 \pgfmathtruncatemacro\k{9-\i} 
 \pgfmathtruncatemacro\l{mod(9+\i,12)}   
 \draw[dashed] (2;\k)--(2;\l);} 
 
\fill [black] (0;0) circle (3.5pt) node[above=2pt, red] {3};

 \foreach \n in {0,1,...,5}{%
   \draw (1;\n) circle (0pt) node[above=2pt, red] {2};}
 \foreach \n in {0,1,...,11}{%
   \draw (2;\n) circle (0pt) node[above=2pt, red] {1};}

\end{tikzpicture}
}
        \caption{}
    \end{subfigure}
    
    \begin{subfigure}{0.3\textwidth}
    \centering
    \resizebox{1.5in}{!}{
\begin{tikzpicture}
\begin{scope}[shift={(-1.5,-.86)}]
    \foreach \row in {0, 1,...,3} {
        \draw[thick] ($\row*(0.5, {0.5*sqrt(3)})$) -- ($(3,0)+\row*(-0.5, {0.5*sqrt(3)})$);
        \draw[thick] ($\row*(1, 0)$) -- ($(3/2,{3/2*sqrt(3)})+\row*(0.5,{-0.5*sqrt(3)})$);
        \draw[thick] ($\row*(1, 0)$) -- ($(0,0)+\row*(0.5,{0.5*sqrt(3)})$);
    }
\end{scope}
\coordinate (0;0) at (0,0); 
\foreach \c in {1,...,2}{%
\foreach \i in {0,...,5}{%
\pgfmathtruncatemacro\j{\c*\i}
\coordinate (\c;\j) at (60*\i:\c);  
} }
\foreach \i in {0,2,...,10}{%
\pgfmathtruncatemacro\j{mod(\i+2,12)}
\pgfmathtruncatemacro\k{\i+1}
\coordinate (2;\k) at ($(2;\i)!.5!(2;\j)$) ;}

 \foreach \i in {0,...,6}{%
 \pgfmathtruncatemacro\k{\i}
 \pgfmathtruncatemacro\l{10-\i}
 \draw[dashed] (2;\k)--(2;\l); 
 \pgfmathtruncatemacro\k{6-\i} 
 \pgfmathtruncatemacro\l{mod(8+\i,12)}   
 \draw[dashed] (2;\k)--(2;\l); 
 \pgfmathtruncatemacro\k{9-\i} 
 \pgfmathtruncatemacro\l{mod(9+\i,12)}   
 \draw[dashed] (2;\k)--(2;\l);} 
 
\fill [black] (0;0) circle (3.5pt) node[above=2pt, red] {3};

 \foreach \n in {0,1,...,5}{%
   \draw (1;\n) circle (0pt) node[above=2pt, red] {2};}
 \foreach \n in {0,1,...,11}{%
   \draw (2;\n) circle (0pt) node[above=2pt, red] {1};}

\end{tikzpicture}    
}
\caption{}
    \end{subfigure}
    \end{tabular}
    }
\caption{$T_1$, $T_2$, and $T_3$ lying within the reach of a dominator with $t=3$.}
    \label{fig:reach of a dominator with t=3}
\end{figure}
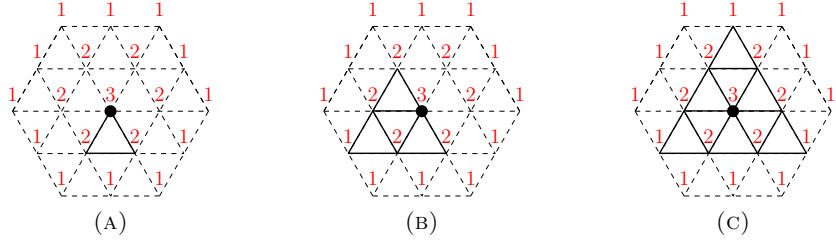
Figure~\ref{fig:reach of a dominator with t=3} provides an example of the reach of a dominator when $t=3$, on which we overlay the graphs $T_1, T_2,$ and $T_3$. Note that placing $T_4$ within the reach of a dominator with $t=3$,  results in some of the vertices of $T_4$ to fall outside of the reach of the dominator. Thus $\gamma_{3,1}(T_4)>1$. This is a case we consider in the following result.

\begin{lemma}\label{lem:medium n t=1}
If $\lfloor\frac{3(t-1)}{2}\rfloor<n\leq 2(t-1)$, then $\gamma_{t,1}(T_n)=2$.
\end{lemma}
\begin{proof}
Suppose that $\lfloor\frac{3(t-1)}{2}\rfloor<n\leq 2(t-1)$. First observe that we can  place the bottom edge of $T_n$ along the main diagonal  on the hexagononal reach of a dominator as illustrated in Figure~\ref{fig:PicNeeded3}. So all vertices within $T_n$ inside of the hexagonal reach of this dominato are receiving reception 1. Then we need only one more dominator to dominate the top portion of $T_n$ outside of reach of the first dominator. Thus $\gamma_{t,1}(T_n)=2$.
\end{proof}

\begin{figure}[h]
\centering
\begin{subfigure}[t]{.45\textwidth}
\centering
\resizebox{!}{1.5in}{%
\resizebox{2.5cm}{!}{%
\begin{tikzpicture}
\begin{scope}[shift={(2,0)}]
\draw[thick] (0:\R) \foreach \x in {60,120,...,359} {
                -- (\x:\R)
            }-- cycle (90:\R);
\fill (0:0) circle (3pt);
\end{scope}

\draw[line width=.6mm, red] (0,0)--(2,3.46)--(4,0)--cycle;

\begin{scope}[shift={(2.25,3.89)}]
\draw[thick] (0:\R) \foreach \x in {60,120,...,359} {
                -- (\x:\R)
            }-- cycle (90:\R);
\fill (0:0) circle (3pt);
\end{scope}
\end{tikzpicture}
}}
\caption{$T_n$ lies within the reach of 2 dominators, when $\lfloor\frac{3(t-1)}{2}\rfloor<n\leq 2(t-1)$.}\label{fig:PicNeeded3}
\end{subfigure}
\hspace{.2in}
\begin{subfigure}[t]{.45\textwidth}
\centering
\resizebox{!}{1.5in}{%
\resizebox{4cm}{!}{%
\begin{tikzpicture}
\begin{scope}[shift={(2,0)}]
\draw[thick] (0:\R) \foreach \x in {60,120,...,359} {
                -- (\x:\R)
            }-- cycle (90:\R);
\fill (0:0) circle (3pt);
\end{scope}

\begin{scope}[shift={(2.25,3.89)}]
\draw[thick] (0:\R) \foreach \x in {60,120,...,359} {
                -- (\x:\R)
            }-- cycle (90:\R);
\fill (0:0) circle (3pt);
\end{scope}

\begin{scope}[shift={(5.5,1.73)}]
\draw[thick] (0:\R) \foreach \x in {60,120,...,359} {
                -- (\x:\R)
            }-- cycle (90:\R);
\fill (0:0) circle (3pt);
\end{scope}

\draw[line width=.6mm, red] (0,0)--(2,3.46)--(4,0)--cycle;
\draw[line width=.6mm, red] (2,3.46)--(2.25,3.89)--(4.5,0)--(4,0);
\node[red, rotate=60] at (2.75,4.76) {\bf\LARGE$\ldots$};
\draw[line width=.6mm, red] (3.25,5.62)--(6.5,0);
\node[red, rotate=0] at (5.25,.5) {{\bf\LARGE $\ldots$}};
\end{tikzpicture}
}}
\caption{$T_{n}$ lies within the reach of 3 dominators, when $2(t-1)<n\leq 3t-2$.}
    \label{fig:3 dominators needed}
    \end{subfigure}
\caption{Figures used in Lemmas \ref{lem:medium n t=1} and \ref{lem:large n t=1}.}
\end{figure}
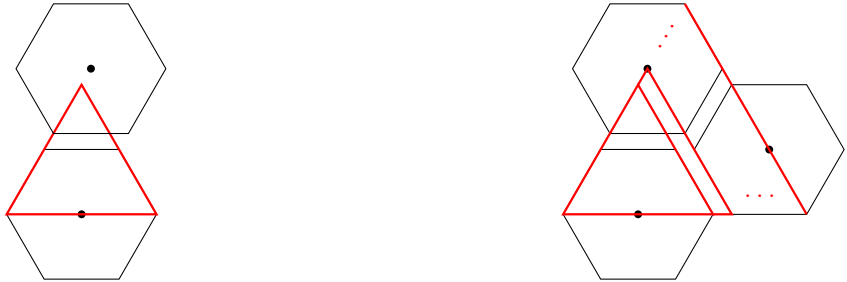

\begin{lemma}\label{lem:large n t=1}
If $2(t-1)<n\leq 3t-2$, then $\gamma_{t,1}(T_n)=3$.
\end{lemma}
\begin{proof}
Arranging 3 dominators as in Figure \ref{fig:3 dominators needed} shows that $\gamma_{t,1}(T_n)\leq 3$ whenever $2(t-1)<n\leq 3t-2$. It now suffices to show that $\gamma_{t,1}(T_n)\neq 2$. To do so we argue based on the the fact that the diameter of the reach of a dominator is $2(t-1)$. Since $n>2(t-1)$ and the triangle has 3 sides we cannot dominate each side of $T_n$   with only $2$ dominators.
\end{proof}

Figure \ref{fig:31dominatingT5T6T7} illustrates the result of Lemma \ref{lem:large n t=1}, when $t=3$. Note that whenever a dominator is placed outside of $T_n$ it suffices to push that dominating vertex to the left so that it lies on~$T_n$. However, we place some dominators outside of $T_n$ to simplify some of the figures. We also highlight the reach of a dominator in purple.

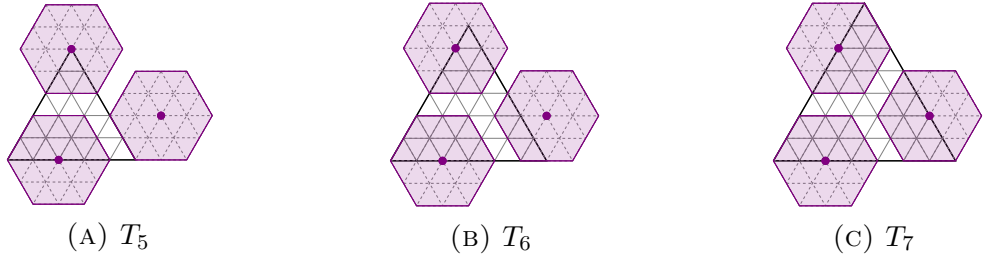
\begin{figure}[H]
     
\begin{tabular}[c]{cc} 
\begin{subfigure}{.3\textwidth}
\centering
    \resizebox{.6\textwidth}{!}{
    
\tikzset{
    hexagon/.pic={
\coordinate (0;0) at (0,0); 
\foreach \c in {1,...,2}{%
\foreach \i in {0,...,5}{%
\pgfmathtruncatemacro\j{\c*\i}
\coordinate (\c;\j) at (60*\i:\c);  
} }
\foreach \i in {0,2,...,10}{%
\pgfmathtruncatemacro\j{mod(\i+2,12)} 
\pgfmathtruncatemacro\k{\i+1}
\coordinate (2;\k) at ($(2;\i)!.5!(2;\j)$) ;}

 \foreach \i in {0,...,6}{%
 \pgfmathtruncatemacro\k{\i}
 \pgfmathtruncatemacro\l{10-\i}
 \draw[gray, dashed] (2;\k)--(2;\l);
 \pgfmathtruncatemacro\k{6-\i} 
 \pgfmathtruncatemacro\l{mod(8+\i,12)}   
 \draw[gray, dashed] (2;\k)--(2;\l); 
 \pgfmathtruncatemacro\k{9-\i} 
 \pgfmathtruncatemacro\l{mod(9+\i,12)}   
 \draw[gray, dashed] (2;\k)--(2;\l);} 
 
\fill [violet] (0;0) circle (4.5pt);

\draw[thick, violet] (0:\R) \foreach \x in {60,120,...,359} {
                -- (\x:\R)
            }-- cycle (90:\R);
            
            \fill [violet] (0;0) circle (4.5pt);

\filldraw[fill=violet!80!violet,opacity=0.15] (0:\R) \foreach \x in {60,120,...,359} {
                -- (\x:\R)
            }-- cycle (90:\R);
\draw[thick, violet] (0:\R) \foreach \x in {60,120,...,359} {
                -- (\x:\R)
            }-- cycle (90:\R);
            }}
            
\begin{tikzpicture}
    \foreach \row in {0, 1,...,5} {
        \draw[gray, thick] ($\row*(0.5, {0.5*sqrt(3)})$) -- ($(5,0)+\row*(-0.5, {0.5*sqrt(3)})$);
        \draw[gray, thick] ($\row*(1, 0)$) -- ($(5/2,{5/2*sqrt(3)})+\row*(0.5,{-0.5*sqrt(3)})$);
        \draw[gray, thick] ($\row*(1, 0)$) -- ($(0,0)+\row*(0.5,{0.5*sqrt(3)})$);
    }
    \draw[black, ultra thick](0,0)--(5,0)--(2.5, 4.33013)--(0,0);
    
\pic at (2,0)  {hexagon};
\pic at (6,1.73) {hexagon};
\pic at (2.5, 4.33) {hexagon};
\end{tikzpicture}
     }
    \caption{$T_5$}
    \label{fig:T_5arranged}
    \end{subfigure}
    \hfill\begin{subfigure}{.3\textwidth}
    \centering
    \resizebox{.6\textwidth}{!}{
    
\tikzset{
    hexagon/.pic={
\coordinate (0;0) at (0,0); 
\foreach \c in {1,...,2}{%
\foreach \i in {0,...,5}{%
\pgfmathtruncatemacro\j{\c*\i}
\coordinate (\c;\j) at (60*\i:\c);  
} }
\foreach \i in {0,2,...,10}{%
\pgfmathtruncatemacro\j{mod(\i+2,12)} 
\pgfmathtruncatemacro\k{\i+1}
\coordinate (2;\k) at ($(2;\i)!.5!(2;\j)$) ;}

 \foreach \i in {0,...,6}{%
 \pgfmathtruncatemacro\k{\i}
 \pgfmathtruncatemacro\l{10-\i}
 \draw[gray, dashed] (2;\k)--(2;\l);
 \pgfmathtruncatemacro\k{6-\i} 
 \pgfmathtruncatemacro\l{mod(8+\i,12)}   
 \draw[gray, dashed] (2;\k)--(2;\l); 
 \pgfmathtruncatemacro\k{9-\i} 
 \pgfmathtruncatemacro\l{mod(9+\i,12)}   
 \draw[gray, dashed] (2;\k)--(2;\l);} 
 
\fill [violet] (0;0) circle (4.5pt);

\draw[thick, violet] (0:\R) \foreach \x in {60,120,...,359} {
                -- (\x:\R)
            }-- cycle (90:\R);

            \fill [violet] (0;0) circle (4.5pt);

\filldraw[fill=violet!80!violet,opacity=0.15] (0:\R) \foreach \x in {60,120,...,359} {
                -- (\x:\R)
            }-- cycle (90:\R);
\draw[thick, violet] (0:\R) \foreach \x in {60,120,...,359} {
                -- (\x:\R)
            }-- cycle (90:\R);
            }}

\begin{tikzpicture}
    \foreach \row in {0, 1,...,6} {
        \draw[gray, thick] ($\row*(0.5, {0.5*sqrt(3)})$) -- ($(6,0)+\row*(-0.5, {0.5*sqrt(3)})$);
        \draw[gray, thick] ($\row*(1, 0)$) -- ($(6/2,{6/2*sqrt(3)})+\row*(0.5,{-0.5*sqrt(3)})$);
        \draw[gray, thick] ($\row*(1, 0)$) -- ($(0,0)+\row*(0.5,{0.5*sqrt(3)})$);
    }
    \draw[black, ultra thick](0,0)--(6,0)--(3., 5.19615)--(0,0);
\pic at (2,0)  {hexagon};
\pic at (6,1.73) {hexagon};
\pic at (2.5, 4.33) {hexagon};
\end{tikzpicture}}\caption{$T_6$}
    \label{fig:T_6}
    \end{subfigure}
    \begin{subfigure}{.3\textwidth}
    \centering
    \resizebox{.6\textwidth}{!}{
    
\tikzset{
    hexagon/.pic={
\coordinate (0;0) at (0,0); 
\foreach \c in {1,...,2}{%
\foreach \i in {0,...,5}{%
\pgfmathtruncatemacro\j{\c*\i}
\coordinate (\c;\j) at (60*\i:\c);  
} }
\foreach \i in {0,2,...,10}{%
\pgfmathtruncatemacro\j{mod(\i+2,12)} 
\pgfmathtruncatemacro\k{\i+1}
\coordinate (2;\k) at ($(2;\i)!.5!(2;\j)$) ;}

 \foreach \i in {0,...,6}{%
 \pgfmathtruncatemacro\k{\i}
 \pgfmathtruncatemacro\l{10-\i}
 \draw[gray, dashed] (2;\k)--(2;\l);
 \pgfmathtruncatemacro\k{6-\i} 
 \pgfmathtruncatemacro\l{mod(8+\i,12)}   
 \draw[gray, dashed] (2;\k)--(2;\l); 
 \pgfmathtruncatemacro\k{9-\i} 
 \pgfmathtruncatemacro\l{mod(9+\i,12)}   
 \draw[gray, dashed] (2;\k)--(2;\l);} 
 
\fill [violet] (0;0) circle (4.5pt);

\draw[thick, violet] (0:\R) \foreach \x in {60,120,...,359} {
                -- (\x:\R)
            }-- cycle (90:\R);
                        
            \fill [violet] (0;0) circle (4.5pt);

\filldraw[fill=violet!80!violet,opacity=0.15] (0:\R) \foreach \x in {60,120,...,359} {
                -- (\x:\R)
            }-- cycle (90:\R);
\draw[thick, violet] (0:\R) \foreach \x in {60,120,...,359} {
                -- (\x:\R)
            }-- cycle (90:\R);
            }}
\begin{tikzpicture}
    \foreach \row in {0, 1,...,7} {
        \draw[gray, thick] ($\row*(0.5, {0.5*sqrt(3)})$) -- ($(7,0)+\row*(-0.5, {0.5*sqrt(3)})$);
        \draw[gray, thick] ($\row*(1, 0)$) -- ($(7/2,{7/2*sqrt(3)})+\row*(0.5,{-0.5*sqrt(3)})$);
        \draw[gray, thick] ($\row*(1, 0)$) -- ($(0,0)+\row*(0.5,{0.5*sqrt(3)})$);
    }
    \draw[black, ultra thick](0,0)--(7,0)--(3.5, 6.06218)--(0,0);
\pic at (2,0)  {hexagon};
\pic at (6,1.73) {hexagon};
\pic at (2.5, 4.33) {hexagon};
\end{tikzpicture}}
    \caption{$T_7$}
    \label{fig:T_7}
    \end{subfigure}
    \end{tabular}
    \caption{Examples of $(3,1)$ broadcast dominating sets for $T_5, T_6, \text{and}\ T_7$.}\label{fig:31dominatingT5T6T7}
\end{figure}

We now provide an upper bound for the $(t,1)$ broadcast domination number of $T_n$ when $t$ is odd.
\begin{theorem}\label{thm:NEWdomnumberfornodd}
If $\ell> 1$ is an integer and $t\geq 2$ is odd, then $\gamma_{t,1}(T_{\ell(2(t-1)-\lfloor\frac{t}{2}\rfloor)})\leq\Delta_{\ell}$.
\end{theorem}

\begin{proof}
Given $t$ odd we can always arrange $\Delta_\ell$ hexagon with a common point of intersection using the midpoints of three edges, as depicted in Figure \ref{fig:upperbound}. We note that the largest $T_n$ that fits within this arrangement of dominators has its base $\lfloor\frac{t}{2}\rfloor$ units below the dominators on the bottom row. Thus, the base length is given by $\ell(2(t-1)-\lfloor\frac{t}{2}\rfloor)$. Thus $\gamma_{t,1}(T_{\ell(2(t-1)-\lfloor\frac{t}{2}\rfloor)})\leq~\Delta_{\ell}$.
\end{proof}

\begin{figure}[H]
    \centering
     \resizebox{3in}{!}{%
 \input{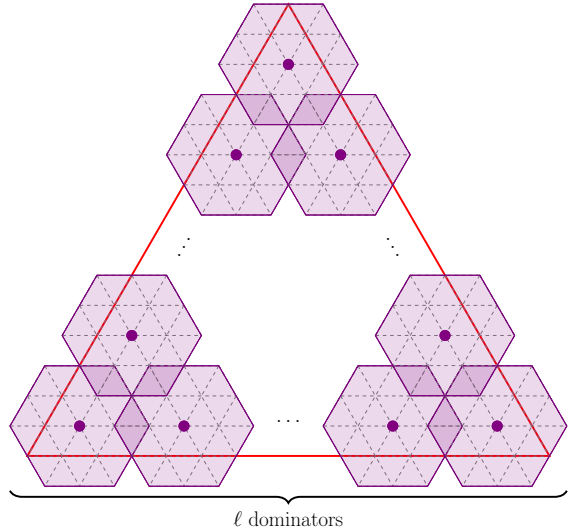}}
    \caption{Placing $T_n$ within the reach of dominators (with $t=3$).}
    \label{fig:upperbound}
\end{figure}

Theorem \ref{thm:NEWdomnumberfornodd} does not provide a tight upper bound. For example when $\ell=3$ and $t=3$, the results states that $\gamma_{3,1}(T_{9})\leq \Delta_3=6$, but in fact we can use less dominators, see Figure~\ref{fig:counterexample}.

\begin{figure}[h]
    \centering
     \resizebox{.25\textwidth}{!}{
     \input{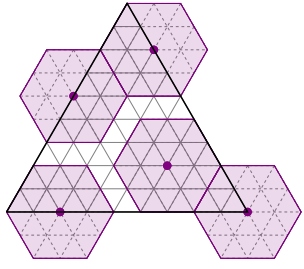}}
    \caption{Example establishing that $\gamma_{3,1}(T_9)\leq5<6=\Delta_{3}$.}\label{fig:counterexample}
\end{figure}

By specializing to $t=3$ we can give a better upper bound for  $\gamma_{3,1}(T_n)$.

\begin{theorem}\label{thm:NEWdomnumberfort=3}
If $n\geq 1$, then $\gamma_{3,1}(T_{n})\leq\Delta_{\left\lfloor\frac{n}{4}+1\right\rfloor}$.
\end{theorem}

Before presenting a proof of this result notice that by Lemmas \ref{lem:smalln t=1} and \ref{lem:large n t=1}
\begin{align*}
    \gamma_{3,1}(T_1)&=1=\Delta_{1}=\Delta_{\lfloor\frac{1}{4}+1\rfloor},\quad \gamma_{3,1}(T_3)=1=\Delta_{1}=\Delta_{\lfloor\frac{3}{4}+1\rfloor}\\
    \gamma_{3,1}(T_5)&=3=\Delta_{2}=\Delta_{\lfloor\frac{5}{4}+1\rfloor},\quad\gamma_{3,1}(T_7)=3=\Delta_{2}=\Delta_{\lfloor\frac{7}{4}+1\rfloor}.
\end{align*} 
Hence, the upper bound provided in Theorem \ref{thm:NEWdomnumberfort=3}  is tight in some cases, but it is not tight in general. Again, Figure \ref{fig:counterexample} provides an example, as in this case the upper bounds provided by Theorem \ref{thm:NEWdomnumberfornodd} and Theorem \ref{thm:NEWdomnumberfort=3} are equal. However, the upper bound in Theorem \ref{thm:NEWdomnumberfort=3} is overall better than the bound of Theorem \ref{thm:NEWdomnumberfornodd} provided $\ell\geq 5$.

\begin{proof}[Proof of Theorem \ref{thm:NEWdomnumberfort=3}]
Let $T_0$ denote the graph consisting of one single vertex. Then the illustrations in Figure~\ref{basecases} establish that $\gamma_{3,1}(T_n)\leq \Delta_{\lfloor\frac{n}{4}+1\rfloor}$ for all $0\leq n\leq 15$. Note that whenever the dominators fall outside of the graph $T_n$, it suffices to shift them to its nearest boundary. However, at times we continue to display the right most dominators outside of $T_n$ for sake of simplicity in our diagrams.

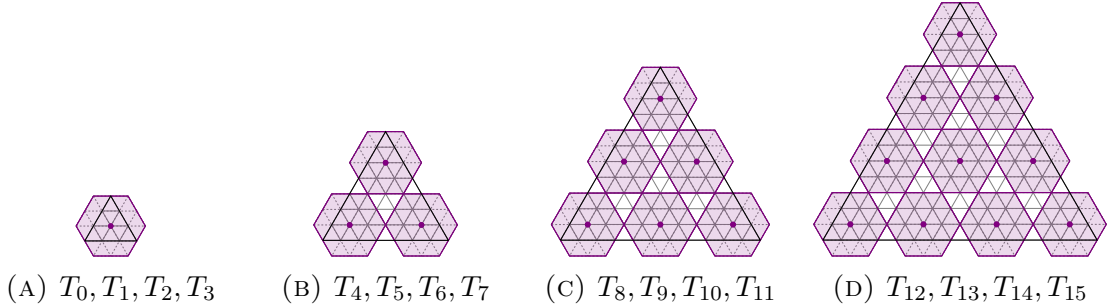
\begin{figure}[h]
\tikzset{
    hexagon/.pic={
\coordinate (0;0) at (0,0); 
\foreach \c in {1,...,2}{%
\foreach \i in {0,...,5}{%
\pgfmathtruncatemacro\j{\c*\i}
\coordinate (\c;\j) at (60*\i:\c);  
} }
\foreach \i in {0,2,...,10}{%
\pgfmathtruncatemacro\j{mod(\i+2,12)} 
\pgfmathtruncatemacro\k{\i+1}
\coordinate (2;\k) at ($(2;\i)!.5!(2;\j)$) ;}

 \foreach \i in {0,...,6}{%
 \pgfmathtruncatemacro\k{\i}
 \pgfmathtruncatemacro\l{10-\i}
 \draw[gray, dashed] (2;\k)--(2;\l);
 \pgfmathtruncatemacro\k{6-\i} 
 \pgfmathtruncatemacro\l{mod(8+\i,12)}   
 \draw[gray, dashed] (2;\k)--(2;\l); 
 \pgfmathtruncatemacro\k{9-\i} 
 \pgfmathtruncatemacro\l{mod(9+\i,12)}   
 \draw[gray, dashed] (2;\k)--(2;\l);} 
 
\fill [violet] (0;0) circle (4.5pt);

\draw[thick, violet] (0:\R) \foreach \x in {60,120,...,359} {
                -- (\x:\R)
            }-- cycle (90:\R);
                       \fill [violet] (0;0) circle (4.5pt);

\filldraw[fill=violet!80!violet,opacity=0.15] (0:\R) \foreach \x in {60,120,...,359} {
                -- (\x:\R)
            }-- cycle (90:\R);
\draw[thick, violet] (0:\R) \foreach \x in {60,120,...,359} {
                -- (\x:\R)
            }-- cycle (90:\R);
            }}
    
 \begin{tabular}[c]{cccc} 
\begin{subfigure}[t]{.2\textwidth}
\centering
    \resizebox{.3\textwidth}{!}{
\begin{tikzpicture}
    \foreach \row in {0, 1,...,3} {
        \draw[gray] ($\row*(0.5, {0.5*sqrt(3)})$) -- ($(3,0)+\row*(-0.5, {0.5*sqrt(3)})$);
        \draw[gray] ($\row*(1, 0)$) -- ($(3/2,{3/2*sqrt(3)})+\row*(0.5,{-0.5*sqrt(3)})$);
        \draw[gray] ($\row*(1, 0)$) -- ($(0,0)+\row*(0.5,{0.5*sqrt(3)})$);
    }
	\pic at (1.5,.86) {hexagon}; 
	\draw[black, ultra thick](0,0)--(3,0)--(1.5, 2.59808)--(0,0);
\end{tikzpicture} 
     }
    \caption{$T_0, T_1,T_2,T_3$}
    \label{fig:G_0,1,2,3}
    \end{subfigure}
&
\begin{subfigure}[t]{.2\textwidth}
\centering
    \resizebox{.6\textwidth}{!}{
\begin{tikzpicture}
    \foreach \row in {0, 1,...,7} {
        \draw[gray] ($\row*(0.5, {0.5*sqrt(3)})$) -- ($(7,0)+\row*(-0.5, {0.5*sqrt(3)})$);
        \draw[gray] ($\row*(1, 0)$) -- ($(7/2,{7/2*sqrt(3)})+\row*(0.5,{-0.5*sqrt(3)})$);
        \draw[gray] ($\row*(1, 0)$) -- ($(0,0)+\row*(0.5,{0.5*sqrt(3)})$);
    }
	\pic at (1.5,.86) {hexagon};
	\pic at (5.5,.86) {hexagon};
	\pic at (3.5,4.33) {hexagon};
	\draw[black, ultra thick](0,0)--(7,0)--(3.5, 6.06218)--(0,0);
\end{tikzpicture}
     }
\caption{$T_4,T_5,T_6,T_7$}
\label{fig:G_4,5,6,7}
\end{subfigure}
&
\begin{subfigure}[t]{.2\textwidth}
\centering
    \resizebox{.9\textwidth}{!}{
\begin{tikzpicture}
    \foreach \row in {0, 1,...,11} {
        \draw[gray] ($\row*(0.5, {0.5*sqrt(3)})$) -- ($(11,0)+\row*(-0.5, {0.5*sqrt(3)})$);
        \draw[gray] ($\row*(1, 0)$) -- ($(11/2,{11/2*sqrt(3)})+\row*(0.5,{-0.5*sqrt(3)})$);
        \draw[gray] ($\row*(1, 0)$) -- ($(0,0)+\row*(0.5,{0.5*sqrt(3)})$);
    }
	\pic at (1.5,.86) {hexagon};
	\pic at (5.5,.86) {hexagon};
	\pic at (9.5,.86) {hexagon};
	\pic at (3.5,4.33) {hexagon};
	\pic at (7.5,4.33) {hexagon};
	\pic at (5.5,7.79) {hexagon};
	\draw[black, ultra thick](0,0)--(11,0)--(5.5, 9.52628)--(0,0);
\end{tikzpicture} 
     }
    \caption{$T_8,T_9,T_{10},T_{11}$}
    \label{fig:G_8,9,10,11}
    \end{subfigure}
    &
\begin{subfigure}[t]{.24\textwidth}
\resizebox{1\textwidth}{!}{
\begin{tikzpicture}
    \foreach \row in {0, 1,...,15} {
        \draw[gray] ($\row*(0.5, {0.5*sqrt(3)})$) -- ($(15,0)+\row*(-0.5, {0.5*sqrt(3)})$);
        \draw[gray] ($\row*(1, 0)$) -- ($(15/2,{15/2*sqrt(3)})+\row*(0.5,{-0.5*sqrt(3)})$);
        \draw[gray] ($\row*(1, 0)$) -- ($(0,0)+\row*(0.5,{0.5*sqrt(3)})$);
    }
	\pic at (1.5,.86) {hexagon};
	\pic at (5.5,.86) {hexagon};
	\pic at (9.5,.86) {hexagon};
	\pic at (13.5,.86) {hexagon};
	\pic at (3.5,4.33) {hexagon};
	\pic at (7.5,4.33) {hexagon};
	\pic at (11.5,4.33) {hexagon};
	\pic at (5.5,7.79) {hexagon};
	\pic at (9.5,7.79) {hexagon};
	\pic at (7.5,11.26) {hexagon};
	\draw[black, ultra thick](0,0)--(15,0)--(7.5, 12.9904)--(0,0);
	\end{tikzpicture} 
     }
    \caption{$T_{12},T_{13},T_{14},T_{15}$}
    \label{fig:G_12,13,14,15}
    \end{subfigure}
 \end{tabular}
\caption{Bases cases for Theorem \ref{thm:NEWdomnumberfort=3}.}\label{basecases}
\end{figure}

Assume, for induction, that $\gamma_{3,1}(T_{k})\leq\Delta_{\left\lfloor\frac{k}{4}+1\right\rfloor}$ for all $0\leq k\leq n$. We have 4  cases to consider depending on whether $n\equiv 0,1,2,3\pmod{4}$.
If $n\equiv0\pmod{4}$, then we arrange the $\Delta_{\lfloor\frac{n}{4}+1\rfloor}$ dominators in $T_n$ as in Figure \ref{fig:0mod4A}. Now by placing $T_{n+1}$ as depicted in Figure \ref{fig:0mod4B}, we note that $\gamma_{3,1}(T_{n+1})\leq\Delta_{\lfloor\frac{n}{4}+1\rfloor}=\Delta_{\lfloor\frac{n+1}{4}+1\rfloor}$, the last equality holding since $n\equiv 0\pmod{4}$.

\begin{figure}[h!]
    \centering
    \begin{tabular}[c]{cc}
    \begin{subfigure}{.3\textwidth}
    \resizebox{1\textwidth}{!}{
    \tikzset{
    hex/.pic={
\coordinate (0;0) at (0,0); 
\foreach \c in {1,...,2}{%
\foreach \i in {0,...,5}{%
\pgfmathtruncatemacro\j{\c*\i}
\coordinate (\c;\j) at (60*\i:\c);  
} }
\foreach \i in {0,2,...,10}{%
\pgfmathtruncatemacro\j{mod(\i+2,12)} 
\pgfmathtruncatemacro\k{\i+1}
\coordinate (2;\k) at ($(2;\i)!.5!(2;\j)$) ;}

 \foreach \i in {0,...,6}{%
 \pgfmathtruncatemacro\k{\i}
 \pgfmathtruncatemacro\l{10-\i}
 \draw[gray] (2;\k)--(2;\l); 
 \pgfmathtruncatemacro\k{6-\i} 
 \pgfmathtruncatemacro\l{mod(8+\i,12)}   
 \draw[gray] (2;\k)--(2;\l); 
 \pgfmathtruncatemacro\k{9-\i} 
 \pgfmathtruncatemacro\l{mod(9+\i,12)}   
 \draw[gray] (2;\k)--(2;\l);} 
 
\fill [black] (0;0) circle (3.5pt);

\draw[line width=.75mm, violet] (0:\R) \foreach \x in {60,120,...,359} {
                -- (\x:\R)
            }-- cycle (90:\R);

            \fill [violet] (0;0) circle (4.5pt);

\filldraw[fill=violet!80!violet,opacity=0.15] (0:\R) \foreach \x in {60,120,...,359} {
                -- (\x:\R)
            }-- cycle (90:\R);
\draw[thick, violet] (0:\R) \foreach \x in {60,120,...,359} {
                -- (\x:\R)
            }-- cycle (90:\R);
            }}
\begin{tikzpicture}
\draw[line width=.75mm, red] (0,0)--(8,0);
\node[red] at (9,0) {{\bf\LARGE$\ldots$}};
\draw[line width=.75mm, red] (10,0)--(12,0);
\draw[line width=.75mm, red] (0,0)--(4,6.93);
\node[red, rotate=60] at (4.5,7.79) {{\bf\LARGE$\ldots$}};
\draw[line width=.75mm, red] (5,8.66)--(6,10.39)--(7,8.66);
\node[red, rotate=-60] at (7.5,7.79) {{\bf\LARGE$\ldots$}};
\draw[line width=.75mm, red] (8,6.93)--(12,0);
\pic at (1.5,.866) {hex};
\pic at (5.5,.866) {hex};
\pic at (13.5,.866) {hex};
\pic at (3.5, 4.33) {hex};
\pic at (7.5,11.25) {hex};
\end{tikzpicture}}
    \caption{}\label{fig:0mod4A}
    \end{subfigure}
    &
    \begin{subfigure}{.3\textwidth}
    \resizebox{0.95\textwidth}{!}{
    \tikzset{
    hex/.pic={
\coordinate (0;0) at (0,0); 
\foreach \c in {1,...,2}{%
\foreach \i in {0,...,5}{%
\pgfmathtruncatemacro\j{\c*\i}
\coordinate (\c;\j) at (60*\i:\c);  
} }
\foreach \i in {0,2,...,10}{%
\pgfmathtruncatemacro\j{mod(\i+2,12)} 
\pgfmathtruncatemacro\k{\i+1}
\coordinate (2;\k) at ($(2;\i)!.5!(2;\j)$) ;}

 \foreach \i in {0,...,6}{%
 \pgfmathtruncatemacro\k{\i}
 \pgfmathtruncatemacro\l{10-\i}
 \draw[gray] (2;\k)--(2;\l); 
 \pgfmathtruncatemacro\k{6-\i} 
 \pgfmathtruncatemacro\l{mod(8+\i,12)}   
 \draw[gray] (2;\k)--(2;\l); 
 \pgfmathtruncatemacro\k{9-\i} 
 \pgfmathtruncatemacro\l{mod(9+\i,12)}   
 \draw[gray] (2;\k)--(2;\l);} 
 
\fill [black] (0;0) circle (3.5pt);

\draw[line width=.75mm, violet] (0:\R) \foreach \x in {60,120,...,359} {
                -- (\x:\R)
            }-- cycle (90:\R);

            \fill [violet] (0;0) circle (4.5pt);

\filldraw[fill=violet!80!violet,opacity=0.15] (0:\R) \foreach \x in {60,120,...,359} {
                -- (\x:\R)
            }-- cycle (90:\R);
\draw[thick, violet] (0:\R) \foreach \x in {60,120,...,359} {
                -- (\x:\R)
            }-- cycle (90:\R);
            }}
\begin{tikzpicture}
\draw[line width=.75mm, red] (0,0)--(8,0);
\node[red] at (9,0) {{\bf\LARGE$\ldots$}};
\draw[line width=.75mm, red] (10,0)--(12,0);
\draw[line width=.75mm, red] (0,0)--(4,6.93);
\node[red, rotate=60] at (4.5,7.79) {{\bf\LARGE$\ldots$}};
\draw[line width=.75mm, red] (5,8.66)--(6,10.39)--(7,8.66);
\node[red, rotate=-60] at (7.5,7.79) {{\bf\LARGE$\ldots$}};
\draw[line width=.75mm, red] (8,6.93)--(12,0);
\draw[line width=.75mm, red] (6,10.39)--(6.5,11.26)--(13,0)--(12,0);
\pic at (1.5,.866) {hex};
\pic at (5.5,.866) {hex};
\pic at (13.5,.866) {hex};
\pic at (3.5, 4.33) {hex};
\pic at (7.5,11.25) {hex};

\end{tikzpicture}}
    \caption{}\label{fig:0mod4B}
    \end{subfigure}
    \end{tabular}
    \caption{$T_n$ with $n\equiv0\pmod{4}$.}
    \label{fig:0mod4}
\end{figure}

If $n\equiv1\pmod{4}$, then we arrange the $\Delta_{\lfloor\frac{n}{4}+1\rfloor}$ dominators in $T_n$ as in Figure \ref{fig:1mod4A}. Now by placing $T_{n+1}$ as depicted in Figure \ref{fig:1mod4B}, we note that $\gamma_{3,1}(T_{n+1})\leq\Delta_{\lfloor\frac{n}{4}+1\rfloor}=\Delta_{\lfloor\frac{n+1}{4}+1\rfloor}$, the last equality holding since $n\equiv 1\pmod{4}$.

\begin{figure}[h!]
    \centering
    \begin{tabular}[c]{cc}
    \begin{subfigure}{.3\textwidth}
    \resizebox{1\textwidth}{!}{
    \tikzset{
    hex/.pic={
\coordinate (0;0) at (0,0); 
\foreach \c in {1,...,2}{%
\foreach \i in {0,...,5}{%
\pgfmathtruncatemacro\j{\c*\i}
\coordinate (\c;\j) at (60*\i:\c);  
} }
\foreach \i in {0,2,...,10}{%
\pgfmathtruncatemacro\j{mod(\i+2,12)} 
\pgfmathtruncatemacro\k{\i+1}
\coordinate (2;\k) at ($(2;\i)!.5!(2;\j)$) ;}

 \foreach \i in {0,...,6}{%
 \pgfmathtruncatemacro\k{\i}
 \pgfmathtruncatemacro\l{10-\i}
 \draw[gray] (2;\k)--(2;\l); 
 \pgfmathtruncatemacro\k{6-\i} 
 \pgfmathtruncatemacro\l{mod(8+\i,12)}   
 \draw[gray] (2;\k)--(2;\l); 
 \pgfmathtruncatemacro\k{9-\i} 
 \pgfmathtruncatemacro\l{mod(9+\i,12)}   
 \draw[gray] (2;\k)--(2;\l);} 
 
\fill [black] (0;0) circle (3.5pt);

\draw[line width=.75mm, violet] (0:\R) \foreach \x in {60,120,...,359} {
                -- (\x:\R)
            }-- cycle (90:\R);
\filldraw[fill=violet!80!violet,opacity=0.15] (0:\R) \foreach \x in {60,120,...,359} {
                -- (\x:\R)
            }-- cycle (90:\R);

}}
\begin{tikzpicture}
\draw[line width=.75mm, red] (0,0)--(8,0);
\node[red] at (9,0) {{\bf\LARGE$\ldots$}};
\draw[line width=.75mm, red] (10,0)--(12,0);
\draw[line width=.75mm, red] (0,0)--(4,6.93);
\node[red, rotate=60] at (4.5,7.79) {{\bf\LARGE$\ldots$}};
\draw[line width=.75mm, red] (5,8.66)--(6,10.39)--(8,6.93);
\node[red, rotate=-60] at (8.5,6.06) {{\bf\LARGE$\ldots$}};
\draw[line width=.75mm, red] (9,5.2)--(12,0);
\pic at (1.5,.866) {hex};
\pic at (5.5,.866) {hex};
\pic at (12.5,.866) {hex};
\pic at (3.5, 4.33) {hex};
\pic at (7,10.39) {hex};
\end{tikzpicture}}
    \caption{}\label{fig:1mod4A}
    \end{subfigure}
    &
    \begin{subfigure}{.3\textwidth}
    \resizebox{0.95\textwidth}{!}{
    \tikzset{
    hex/.pic={
\coordinate (0;0) at (0,0); 
\foreach \c in {1,...,2}{%
\foreach \i in {0,...,5}{%
\pgfmathtruncatemacro\j{\c*\i}
\coordinate (\c;\j) at (60*\i:\c);  
} }
\foreach \i in {0,2,...,10}{%
\pgfmathtruncatemacro\j{mod(\i+2,12)} 
\pgfmathtruncatemacro\k{\i+1}
\coordinate (2;\k) at ($(2;\i)!.5!(2;\j)$) ;}

 \foreach \i in {0,...,6}{%
 \pgfmathtruncatemacro\k{\i}
 \pgfmathtruncatemacro\l{10-\i}
 \draw[gray] (2;\k)--(2;\l); 
 \pgfmathtruncatemacro\k{6-\i} 
 \pgfmathtruncatemacro\l{mod(8+\i,12)}   
 \draw[gray] (2;\k)--(2;\l); 
 \pgfmathtruncatemacro\k{9-\i} 
 \pgfmathtruncatemacro\l{mod(9+\i,12)}   
 \draw[gray] (2;\k)--(2;\l);} 
 
\fill [black] (0;0) circle (3.5pt);

\draw[line width=.75mm, violet] (0:\R) \foreach \x in {60,120,...,359} {
                -- (\x:\R)
            }-- cycle (90:\R);
\filldraw[fill=violet!80!violet,opacity=0.15] (0:\R) \foreach \x in {60,120,...,359} {
                -- (\x:\R)
            }-- cycle (90:\R);

}}
\begin{tikzpicture}
\draw[line width=.75mm, red] (0,0)--(8,0);
\node[red] at (9,0) {{\bf\LARGE$\ldots$}};
\draw[line width=.75mm, red] (10,0)--(12,0);
\draw[line width=.75mm, red] (0,0)--(4,6.93);
\node[red, rotate=60] at (4.5,7.79) {{\bf\LARGE$\ldots$}};
\draw[line width=.75mm, red] (5,8.66)--(6,10.39)--(8,6.93);
\node[red, rotate=-60] at (8.5,6.06) {{\bf\LARGE$\ldots$}};
\draw[line width=.75mm, red] (9,5.2)--(12,0);
\draw[line width=.75mm, red] (6,10.39)--(6.5,11.26)--(13,0)--(12,0);
\pic at (1.5,.866) {hex};
\pic at (5.5,.866) {hex};
\pic at (12.5,.866) {hex};
\pic at (3.5, 4.33) {hex};
\pic at (7,10.39) {hex};
\end{tikzpicture}}
    \caption{}\label{fig:1mod4B}
    \end{subfigure}
    \end{tabular}
    \caption{$T_n$ with $n\equiv1\pmod{4}$.}
    \label{fig:1mod4}
\end{figure}

If $n\equiv2\pmod{4}$, then we arrange the $\Delta_{\lfloor\frac{n}{4}+1\rfloor}$ dominators in $T_n$ as in Figure \ref{fig:2mod4A}. Now by placing $T_{n+1}$ as depicted in Figure \ref{fig:2mod4B}, we note that $\gamma_{3,1}(T_{n+1})\leq\Delta_{\lfloor\frac{n}{4}+1\rfloor}=\Delta_{\lfloor\frac{n+1}{4}+1\rfloor}$, the last equality holding since $n\equiv 2\pmod{4}$.

\begin{figure}[h!]
    \centering
    \begin{tabular}[c]{cc}
    \begin{subfigure}{.3\textwidth}
    \resizebox{1\textwidth}{!}{
    \tikzset{
    hex/.pic={
\coordinate (0;0) at (0,0); 
\foreach \c in {1,...,2}{%
\foreach \i in {0,...,5}{%
\pgfmathtruncatemacro\j{\c*\i}
\coordinate (\c;\j) at (60*\i:\c);  
} }
\foreach \i in {0,2,...,10}{%
\pgfmathtruncatemacro\j{mod(\i+2,12)} 
\pgfmathtruncatemacro\k{\i+1}
\coordinate (2;\k) at ($(2;\i)!.5!(2;\j)$) ;}

 \foreach \i in {0,...,6}{%
 \pgfmathtruncatemacro\k{\i}
 \pgfmathtruncatemacro\l{10-\i}
 \draw[gray] (2;\k)--(2;\l); 
 \pgfmathtruncatemacro\k{6-\i} 
 \pgfmathtruncatemacro\l{mod(8+\i,12)}   
 \draw[gray] (2;\k)--(2;\l); 
 \pgfmathtruncatemacro\k{9-\i} 
 \pgfmathtruncatemacro\l{mod(9+\i,12)}   
 \draw[gray] (2;\k)--(2;\l);} 
 
\fill [black] (0;0) circle (3.5pt);

\draw[line width=.75mm, violet] (0:\R) \foreach \x in {60,120,...,359} {
                -- (\x:\R)
            }-- cycle (90:\R);

        \fill [violet] (0;0) circle (4.5pt);

\filldraw[fill=violet!80!violet,opacity=0.15] (0:\R) \foreach \x in {60,120,...,359} {
                -- (\x:\R)
            }-- cycle (90:\R);

            }}
\begin{tikzpicture}
\draw[line width=.75mm, red] (0,0)--(8,0);
\node[red] at (9,0) {{\bf\LARGE$\ldots$}};
\draw[line width=.75mm, red] (10,0)--(12,0);
\draw[line width=.75mm, red] (0,0)--(4,6.93);
\node[red, rotate=60] at (4.5,7.79) {{\bf\LARGE${\ldots}$}};
\draw[line width=.75mm, red] (5,8.66)--(6,10.39)--(8,6.93);
\node[red, rotate=-60] at (8.5,6.06) {{\bf\LARGE$\ldots$}};
\draw[line width=.75mm, red] (9,5.2)--(12,0);
\pic at (1.5,.866) {hex};
\pic at (5.5,.866) {hex};
\pic at (11.5,.866) {hex};
\pic at (3.5, 4.33) {hex};
\pic at (6.5,9.53) {hex};
\end{tikzpicture}}
    \caption{}\label{fig:2mod4A}
    \end{subfigure}
    &
    \begin{subfigure}{.3\textwidth}
    \resizebox{0.95\textwidth}{!}{
    \tikzset{
    hex/.pic={
\coordinate (0;0) at (0,0); 
\foreach \c in {1,...,2}{%
\foreach \i in {0,...,5}{%
\pgfmathtruncatemacro\j{\c*\i}
\coordinate (\c;\j) at (60*\i:\c);  
} }
\foreach \i in {0,2,...,10}{%
\pgfmathtruncatemacro\j{mod(\i+2,12)} 
\pgfmathtruncatemacro\k{\i+1}
\coordinate (2;\k) at ($(2;\i)!.5!(2;\j)$) ;}

 \foreach \i in {0,...,6}{%
 \pgfmathtruncatemacro\k{\i}
 \pgfmathtruncatemacro\l{10-\i}
 \draw[gray] (2;\k)--(2;\l); 
 \pgfmathtruncatemacro\k{6-\i} 
 \pgfmathtruncatemacro\l{mod(8+\i,12)}   
 \draw[gray] (2;\k)--(2;\l); 
 \pgfmathtruncatemacro\k{9-\i} 
 \pgfmathtruncatemacro\l{mod(9+\i,12)}   
 \draw[gray] (2;\k)--(2;\l);} 
 
\fill [black] (0;0) circle (3.5pt);

\draw[line width=.75mm, violet] (0:\R) \foreach \x in {60,120,...,359} {
                -- (\x:\R)
            }-- cycle (90:\R);
            \fill [violet] (0;0) circle (4.5pt);

\filldraw[fill=violet!80!violet,opacity=0.15] (0:\R) \foreach \x in {60,120,...,359} {
                -- (\x:\R)
            }-- cycle (90:\R);

            }}
\begin{tikzpicture}
\draw[line width=.75mm, red] (0,0)--(8,0);
\node[red] at (9,0) {{\bf\LARGE$\ldots$}};
\draw[line width=.75mm, red] (10,0)--(12,0);
\draw[line width=.75mm, red] (0,0)--(4,6.93);
\node[red, rotate=60] at (4.5,7.79) {{\bf\LARGE$\ldots$}};
\draw[line width=.75mm, red] (5,8.66)--(6,10.39)--(8,6.93);
\node[red, rotate=-60] at (8.5,6.06) {{\bf\LARGE$\ldots$}};
\draw[line width=.75mm, red] (9,5.2)--(12,0);
\draw[line width=.75mm, red] (6,10.39)--(6.5,11.26)--(13,0)--(12,0);
\pic at (1.5,.866) {hex};
\pic at (5.5,.866) {hex};
\pic at (11.5,.866) {hex};
\pic at (3.5, 4.33) {hex};
\pic at (6.5,9.53) {hex};
\end{tikzpicture}}
    \caption{}\label{fig:2mod4B}
    \end{subfigure}
    \end{tabular}
    \caption{$T_n$ with $n\equiv2\pmod{4}$.}
    \label{fig:2mod4}
\end{figure}
The last case to consider is when $n\equiv3 \pmod{4}$. In this case we arrange the $\Delta_{\lfloor\frac{n}{4}+1\rfloor}$ dominators in $T_n$ as in Figure \ref{fig:3mod4A}. Now by placing $T_{n+1}$ as depicted in Figure \ref{fig:3mod4B}, we note that $\Delta_{\lfloor\frac{n}{4}+1\rfloor}$ are not enough dominators to cover the right most side of $T_{n+1}$. So we add an additional $\lfloor\frac{n}{4}+1\rfloor+1$ dominators as depicted in Figure \ref{fig:3mod4C} to dominate the right hand side of $T_{n+1}$. Thus 
$\gamma_{3,1}(T_{n+1})\leq \Delta_{\lfloor\frac{n}{4}+1\rfloor}+\left(\left\lfloor\frac{n}{4}+1\right\rfloor+1\right)=\Delta_{\lfloor\frac{n}{4}+1\rfloor+1}=\Delta_{\lfloor\frac{n+1}{4}+1\rfloor}$,
the last equality holding since $n\equiv 3\pmod{4}$.
\begin{figure}[h!]
    \centering
    \begin{tabular}[c]{ccc}
    \begin{subfigure}{.3\textwidth}
    \resizebox{1\textwidth}{!}{
    \tikzset{
    hex/.pic={
\coordinate (0;0) at (0,0); 
\foreach \c in {1,...,2}{%
\foreach \i in {0,...,5}{%
\pgfmathtruncatemacro\j{\c*\i}
\coordinate (\c;\j) at (60*\i:\c);  
} }
\foreach \i in {0,2,...,10}{%
\pgfmathtruncatemacro\j{mod(\i+2,12)} 
\pgfmathtruncatemacro\k{\i+1}
\coordinate (2;\k) at ($(2;\i)!.5!(2;\j)$) ;}

 \foreach \i in {0,...,6}{%
 \pgfmathtruncatemacro\k{\i}
 \pgfmathtruncatemacro\l{10-\i}
 \draw[gray] (2;\k)--(2;\l); 
 \pgfmathtruncatemacro\k{6-\i} 
 \pgfmathtruncatemacro\l{mod(8+\i,12)}   
 \draw[gray] (2;\k)--(2;\l); 
 \pgfmathtruncatemacro\k{9-\i} 
 \pgfmathtruncatemacro\l{mod(9+\i,12)}   
 \draw[gray] (2;\k)--(2;\l);} 
 
\fill [black] (0;0) circle (3.5pt);

\draw[line width=.75mm, violet] (0:\R) \foreach \x in {60,120,...,359} {
                -- (\x:\R)
            }-- cycle (90:\R);

            \fill [violet] (0;0) circle (4.5pt);

\filldraw[fill=violet!80!violet,opacity=0.15] (0:\R) \foreach \x in {60,120,...,359} {
                -- (\x:\R)
            }-- cycle (90:\R);
            }}
\begin{tikzpicture}
\draw[line width=.75mm, red] (0,0)--(7.5,0);
\node[red] at (8.5,0) {{\bf\LARGE$\ldots$}};
\draw[line width=.75mm, red] (9.5,0)--(13,0);
\draw[line width=.75mm, red] (0,0)--(3.75,6.5);
\node[red, rotate=60] at (4.25,7.36) {{\bf\LARGE$\ldots$}};
\draw[line width=.75mm, red] (4.75,8.22)--(6.5,11.26)--(9,6.93);
\node[red, rotate=-60] at (9.5,6.06) {{\bf\LARGE$\ldots$}};
\draw[line width=.75mm, red] (10,5.2)--(13,0);
\pic at (1.5,.866) {hex};
\pic at (5.5,.866) {hex};
\pic at (11.5,.866) {hex};
\pic at (3.5, 4.33) {hex};
\pic at (6.5,9.53) {hex};
\end{tikzpicture}}
    \caption{}\label{fig:3mod4A}
    \end{subfigure}
    &
    \begin{subfigure}{.3\textwidth}
    \resizebox{0.95\textwidth}{!}{
    \tikzset{
    hex/.pic={
\coordinate (0;0) at (0,0); 
\foreach \c in {1,...,2}{%
\foreach \i in {0,...,5}{%
\pgfmathtruncatemacro\j{\c*\i}
\coordinate (\c;\j) at (60*\i:\c);  
} }
\foreach \i in {0,2,...,10}{%
\pgfmathtruncatemacro\j{mod(\i+2,12)} 
\pgfmathtruncatemacro\k{\i+1}
\coordinate (2;\k) at ($(2;\i)!.5!(2;\j)$) ;}

 \foreach \i in {0,...,6}{%
 \pgfmathtruncatemacro\k{\i}
 \pgfmathtruncatemacro\l{10-\i}
 \draw[gray] (2;\k)--(2;\l); 
 \pgfmathtruncatemacro\k{6-\i} 
 \pgfmathtruncatemacro\l{mod(8+\i,12)}   
 \draw[gray] (2;\k)--(2;\l); 
 \pgfmathtruncatemacro\k{9-\i} 
 \pgfmathtruncatemacro\l{mod(9+\i,12)}   
 \draw[gray] (2;\k)--(2;\l);} 
 
\fill [black] (0;0) circle (3.5pt);

\draw[line width=.75mm, violet] (0:\R) \foreach \x in {60,120,...,359} {
                -- (\x:\R)
            }-- cycle (90:\R);
\filldraw[fill=violet!80!violet,opacity=0.15] (0:\R) \foreach \x in {60,120,...,359} {
                -- (\x:\R)
            }-- cycle (90:\R);

}}
\begin{tikzpicture}
\draw[line width=.75mm, red] (0,0)--(7.5,0);
\node[red] at (8.5,0) {{\bf\LARGE$\ldots$}};
\draw[line width=.75mm, red] (9.5,0)--(13,0);
\draw[line width=.75mm, red] (0,0)--(3.75,6.5);
\node[red, rotate=60] at (4.25,7.36) {{\bf\LARGE$\ldots$}};
\draw[line width=.75mm, red] (4.75,8.22)--(6.5,11.26)--(9,6.93);
\node[red, rotate=-60] at (9.5,6.06) {{\bf\LARGE$\ldots$}};
\draw[line width=.75mm, red] (10,5.2)--(13,0);
\draw[line width=.75mm, red] (6.5,11.26)--(7,12.12)--(14,0)--(13,0);
\pic at (1.5,.866) {hex};
\pic at (5.5,.866) {hex};
\pic at (11.5,.866) {hex};
\pic at (3.5, 4.33) {hex};
\pic at (6.5,9.53) {hex};
\end{tikzpicture}}
    \caption{}\label{fig:3mod4B}
    \end{subfigure}
    &
    \begin{subfigure}{.3\textwidth}
    \resizebox{0.95\textwidth}{!}{
    \tikzset{
    hex/.pic={
\coordinate (0;0) at (0,0); 
\foreach \c in {1,...,2}{%
\foreach \i in {0,...,5}{%
\pgfmathtruncatemacro\j{\c*\i}
\coordinate (\c;\j) at (60*\i:\c);  
} }
\foreach \i in {0,2,...,10}{%
\pgfmathtruncatemacro\j{mod(\i+2,12)} 
\pgfmathtruncatemacro\k{\i+1}
\coordinate (2;\k) at ($(2;\i)!.5!(2;\j)$) ;}

 \foreach \i in {0,...,6}{%
 \pgfmathtruncatemacro\k{\i}
 \pgfmathtruncatemacro\l{10-\i}
 \draw[gray] (2;\k)--(2;\l); 
 \pgfmathtruncatemacro\k{6-\i} 
 \pgfmathtruncatemacro\l{mod(8+\i,12)}   
 \draw[gray] (2;\k)--(2;\l); 
 \pgfmathtruncatemacro\k{9-\i} 
 \pgfmathtruncatemacro\l{mod(9+\i,12)}   
 \draw[gray] (2;\k)--(2;\l);} 
 
\fill [black] (0;0) circle (3.5pt);

\draw[line width=.75mm, violet] (0:\R) \foreach \x in {60,120,...,359} {
                -- (\x:\R)
            }-- cycle (90:\R);
\filldraw[fill=violet!80!violet,opacity=0.15] (0:\R) \foreach \x in {60,120,...,359} {
                -- (\x:\R)
            }-- cycle (90:\R);

}}
\begin{tikzpicture}
\draw[line width=.75mm, red] (0,0)--(7.5,0);
\node[red] at (8.5,0) {{\bf\LARGE$\ldots$}};
\draw[line width=.75mm, red] (9.5,0)--(13,0);
\draw[line width=.75mm, red] (0,0)--(3.75,6.5);
\node[red, rotate=60] at (4.25,7.36) {{\bf\LARGE$\ldots$}};
\draw[line width=.75mm, red] (4.75,8.22)--(6.5,11.26)--(9,6.93);
\node[red, rotate=-60] at (9.5,6.06) {{\bf\LARGE$\ldots$}};
\draw[line width=.75mm, red] (10,5.2)--(13,0);
\draw[line width=.75mm, red] (6.5,11.26)--(7,12.12)--(14,0)--(13,0);
\pic at (1.5,.866) {hex};
\pic at (5.5,.866) {hex};
\pic at (11.5,.866) {hex};
\pic at (15.5,.866) {hex};
\pic at (3.5, 4.33) {hex};
\pic at (13.5, 4.33) {hex};
\pic at (6.5,9.53) {hex};
\pic at (10.5,9.53) {hex};
\pic at (8.5,12.99) {hex};
\end{tikzpicture}}
    \caption{}\label{fig:3mod4C}
    \end{subfigure}
    \end{tabular}
    \caption{$T_n$ with $n\equiv3\pmod{4}$.}
    \label{fig:3mod4}
\end{figure}
\end{proof}

\section{Domination patterns on the triangular grid}\label{sec:infinitegrid}
In this section, we study $\T$, the infinite triangular grid graph, and construct efficient $(t,r)$ broadcast dominating sets for all $t\geq r\geq 1$ by using the geometric interpretation of a dominator's hexagonal reach. We refer to a set of dominators in an efficient $(t,r)$ broadcast dominating set as a broadcast domination pattern.

We preface the broadcast domination patterns with a definition of how to denote the placement of dominators on $\T$.
Throughout this section, we define the unit vectors $\alpha_1 = (1,0)$ and $\alpha_2 = \left( \frac{-1}{2}, \frac{\sqrt{3}}{2} \right)$ and use these vectors to denote specific vertices in $\T$ starting at a predetermined origin vertex $(0,0)$.

In order to anchor our $\tr$ broadcast domination patterns, we force that the origin $(0,0)$ is always a dominating vertex. With this in mind, we now state our main result.
\begin{theorem} \label{theorem:infinitedomination}
Let $t\geq r\geq 1$. Then an efficient $(t,r)$ broadcast domination pattern for the infinite triangular grid is given by placing dominators at every vertex of the form
\begin{align}
[(2t-r)x+(t-r)y]\alpha_1 + [tx+ (2t-r)y]\alpha_2
\label{thm:equation}
\end{align}
\noindent {with $x,y \in \mathbb{Z}$.} 
\end{theorem}

\begin{proof}
We must show that placing dominators at each of the vertices of \eqref{thm:equation} dominates the infinite triangular grid efficiently. We do so by using the definition of the reach of a dominator and by calculating distances from vertices to dominators. 

Figure \ref{fig:hh} provides a visualization of the vertices reached by dominator $A$ located at $(0,0)$. The green region denotes the location of vertices near $A$, receiving signal at least $r$ from dominator $A$. Those lying on the outer boundary of the yellow region are the vertices receiving $r-1$ signal from dominator $A$, and the vertices lying on the red outer boundary are receiving $1$ signal from dominator $A$. Any other vertex of $\T$ not lying within the red hexagon receive no signal from dominator $A$.

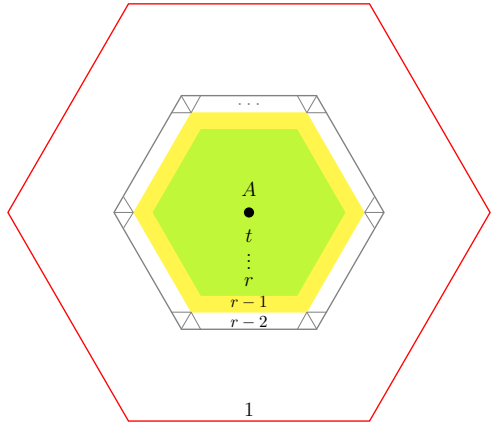
\begin{figure}[H]
     \centering
     \resizebox{.4\textwidth}{!}{
     \begin{tikzpicture}
\draw[thick, red] (0:2.5*\R) \foreach \x in {60,120,...,359} {
                -- (\x:2.5*\R)
            }-- cycle (90:2.5*\R);
\draw[thick, gray] (0:1.4*\R) \foreach \x in {60,120,...,359} {
                -- (\x:1.4*\R)
            }-- cycle (90:1.4*\R);
\fill[yellow, opacity=.7] (0:1.2*\R) \foreach \x in {60,120,...,359} {
                -- (\x:1.2*\R)
            }-- cycle (90:1.2*\R);
\fill[green, nearly transparent] (0:\R) \foreach \x in {60,120,...,359} {
                -- (\x:\R)
            }-- cycle (90:\R);

\draw[gray](2.4,0)--(2.8,0);
\draw[gray](2.4,0)--(2.6,.32);
\draw[gray](2.4,0)--(2.6,-.32);
\draw[gray](1.2,2.07)--(1.41,2.43);
\draw[gray](1.2,2.07)--(1.62,2.07);
\draw[gray](1.2,2.07)--(1,2.43);

\draw[gray](-2.4,0)--(-2.8,0);
\draw[gray](-2.4,0)--(-2.6,.32);
\draw[gray](-2.4,0)--(-2.6,-.32);
\draw[gray](-1.2,2.07)--(-1.41,2.43);
\draw[gray](-1.2,2.07)--(-1.62,2.07);
\draw[gray](-1.2,2.07)--(-1,2.43);

\draw[gray](-1.2,-2.07)--(-1.41,-2.43);
\draw[gray](-1.2,-2.07)--(-1.62,-2.07);
\draw[gray](-1.2,-2.07)--(-1,-2.43);
\draw[gray](1.2,-2.07)--(1.41,-2.43);
\draw[gray](1.2,-2.07)--(1.62,-2.07);
\draw[gray](1.2,-2.07)--(1,-2.43);

\fill [black] (0:0) circle (3pt);
\node at (0:0) [above =2mm]{$A$};
\node at (0:0) [below =2mm]{$t$};
\node at (0,-.5) [below]{$\vdots$};
\node at (0,-1) [below =2mm]{$r$};
\node at (0,-1.4) [below =2mm]{\footnotesize $r-1$};
\node at (0,-1.8) [below =2mm]{\footnotesize $r-2$};
\node at (0,-3.6) [below =2mm]{$1$};
\node[gray] at (0,2.25) {$\ldots$};

\end{tikzpicture}}
     \caption{Generalized reach of a dominator with signal strength $t$. }
     \label{fig:hh}
 \end{figure}

In Figure \ref{fig:dp1} we plot 6 other dominators surrounding $A$, as described by Equation \eqref{thm:equation}:
\begin{align*}
    &B\;\mbox{at}\;t\alpha_1+(-t+r)\alpha_2 \;\mbox{when}\;x=1,y=-1,\\
    &C\;\mbox{at}\;(2t-r)\alpha_1+t\alpha_2 \;\mbox{when}\;x=1,y=0,\\
    &D\;\mbox{at}\;(t-r)\alpha_1+(2t-r)\alpha_2\;\mbox{when}\;x=0,y=1,\\
    &E\;\mbox{at}\;-t\alpha_1+(t-r)\alpha_2 \;\mbox{when}\; x=-1,y=1,\\
    &F\;\mbox{at}\;(-2t+r)\alpha_1-t\alpha_2 \;\mbox{when}\;x=-1,y=0,\\
    &G\;\mbox{at}\;(-t+r)\alpha_1+(-2t+r)\alpha_2 \;\mbox{when}\;x=0,y=-1.
\end{align*}
\begin{figure}[H]
    \centering
    \begin{subfigure}[b]{0.48\textwidth}
     \resizebox{\textwidth}{!}{\tikzset{
    harrishexagon/.pic={
    \draw[thick, red] (0:2.5*\R) \foreach \x in {60,120,...,359} {
                -- (\x:2.5*\R)
            }-- cycle (90:2.5*\R);
    \draw[thick, gray] (0:1.4*\R) \foreach \x in {60,120,...,359} {
                -- (\x:1.4*\R)
            }-- cycle (90:1.4*\R);
    \fill[yellow, opacity=.7] (0:1.2*\R) \foreach \x in {60,120,...,359} {
                -- (\x:1.2*\R)
            }-- cycle (90:1.2*\R);
    \fill[green, nearly transparent] (0:\R) \foreach \x in {60,120,...,359} {
                -- (\x:\R)
            }-- cycle (90:\R);
\draw[gray](2.4,0)--(2.8,0);
\draw[gray](2.4,0)--(2.6,.32);
\draw[gray](2.4,0)--(2.6,-.32);
\draw[gray](1.2,2.07)--(1.41,2.43);
\draw[gray](1.2,2.07)--(1.62,2.07);
\draw[gray](1.2,2.07)--(1,2.43);

\draw[gray](-2.4,0)--(-2.8,0);
\draw[gray](-2.4,0)--(-2.6,.32);
\draw[gray](-2.4,0)--(-2.6,-.32);
\draw[gray](-1.2,2.07)--(-1.41,2.43);
\draw[gray](-1.2,2.07)--(-1.62,2.07);
\draw[gray](-1.2,2.07)--(-1,2.43);

\draw[gray](-1.2,-2.07)--(-1.41,-2.43);
\draw[gray](-1.2,-2.07)--(-1.62,-2.07);
\draw[gray](-1.2,-2.07)--(-1,-2.43);
\draw[gray](1.2,-2.07)--(1.41,-2.43);
\draw[gray](1.2,-2.07)--(1.62,-2.07);
\draw[gray](1.2,-2.07)--(1,-2.43);

\fill [black] (0:0) circle (3pt);
            }}
\resizebox{12cm}{!}{%

\begin{tikzpicture}[scale=0.5]
\begin{scope}[yscale=-1,xscale=1]
\pic at (0,0) {harrishexagon};
\node at (0,0) [below=2mm]{$A$};

\pic at (9.4,-9.35) {harrishexagon};
\node at (9.4,-9.35) [above right=2mm]{$C$};
\pic at (-3.4,-12.8) {harrishexagon};
\node at (-3.4,-12.8) [above=2mm]{$D$};
\pic at (-12.8,-3.5) {harrishexagon};
\node at (-12.8,-3.5) [above left=2mm]{$E$};
\pic at (-9.4,9.35) {harrishexagon};
\node at (-9.4,9.35) [below left=2mm]{$F$};
\pic at (3.4,12.8) {harrishexagon};
\node at (3.4,12.8) [below=2mm]{$G$};
\pic at (12.8,3.5) {harrishexagon};
\node at (12.8,3.5) [below right=2mm]{$B$};
\end{scope}
\end{tikzpicture}
}}
        \caption{Broadcast dominating pattern for $\T$.}
        \label{fig:dp}
    \end{subfigure}
    \hfill
    \begin{subfigure}[b]{0.48\textwidth}
       \resizebox{\textwidth}{!}{ \tikzset{
    harrishexagon/.pic={
    \draw[thick, red] (0:2.5*\R) \foreach \x in {60,120,...,359} {
                -- (\x:2.5*\R)
            }-- cycle (90:2.5*\R);
    \draw[thick, gray] (0:1.4*\R) \foreach \x in {60,120,...,359} {
                -- (\x:1.4*\R)
            }-- cycle (90:1.4*\R);
    \fill[yellow, opacity=.7] (0:1.2*\R) \foreach \x in {60,120,...,359} {
                -- (\x:1.2*\R)
            }-- cycle (90:1.2*\R);
    \fill[green, nearly transparent] (0:\R) \foreach \x in {60,120,...,359} {
                -- (\x:\R)
            }-- cycle (90:\R);
\draw[gray](2.4,0)--(2.8,0);
\draw[gray](2.4,0)--(2.6,.32);
\draw[gray](2.4,0)--(2.6,-.32);
\draw[gray](1.2,2.07)--(1.41,2.43);
\draw[gray](1.2,2.07)--(1.62,2.07);
\draw[gray](1.2,2.07)--(1,2.43);

\draw[gray](-2.4,0)--(-2.8,0);
\draw[gray](-2.4,0)--(-2.6,.32);
\draw[gray](-2.4,0)--(-2.6,-.32);
\draw[gray](-1.2,2.07)--(-1.41,2.43);
\draw[gray](-1.2,2.07)--(-1.62,2.07);
\draw[gray](-1.2,2.07)--(-1,2.43);

\draw[gray](-1.2,-2.07)--(-1.41,-2.43);
\draw[gray](-1.2,-2.07)--(-1.62,-2.07);
\draw[gray](-1.2,-2.07)--(-1,-2.43);
\draw[gray](1.2,-2.07)--(1.41,-2.43);
\draw[gray](1.2,-2.07)--(1.62,-2.07);
\draw[gray](1.2,-2.07)--(1,-2.43);

\fill [black] (0:0) circle (3pt);
            }}
\resizebox{12cm}{!}{%

\begin{tikzpicture}[scale=0.5]
\begin{scope}[yscale=-1,xscale=1]
\pic at (0,0) {harrishexagon};
\node at (0,0) [below=2mm]{$A$};

\pic at (9.4,-9.35) {harrishexagon};
\node at (9.4,-9.35) [above right=2mm]{$C$};
\pic at (-3.4,-12.8) {harrishexagon};
\node at (-3.4,-12.8) [above=2mm]{$D$};
\pic at (-12.8,-3.5) {harrishexagon};
\node at (-12.8,-3.5) [above left=2mm]{$E$};
\pic at (-9.4,9.35) {harrishexagon};
\node at (-9.4,9.35) [below left=2mm]{$F$};
\pic at (3.4,12.8) {harrishexagon};
\node at (3.4,12.8) [below=2mm]{$G$};
\pic at (12.8,3.5) {harrishexagon};
\node at (12.8,3.5) [below right=2mm]{$B$};
\end{scope}
\filldraw[fill=violet!80!violet,opacity=0.3] (4.4,.7)--(9.6,.7)--(5,8.7)--(2.4,4.2)--cycle;
\end{tikzpicture}

}}
        \caption{Shaded region.}
        \label{fig:shadedregion}
    \end{subfigure}
    \caption{Plotting 7 dominators in $\T$.}\label{fig:dp1}
\end{figure}
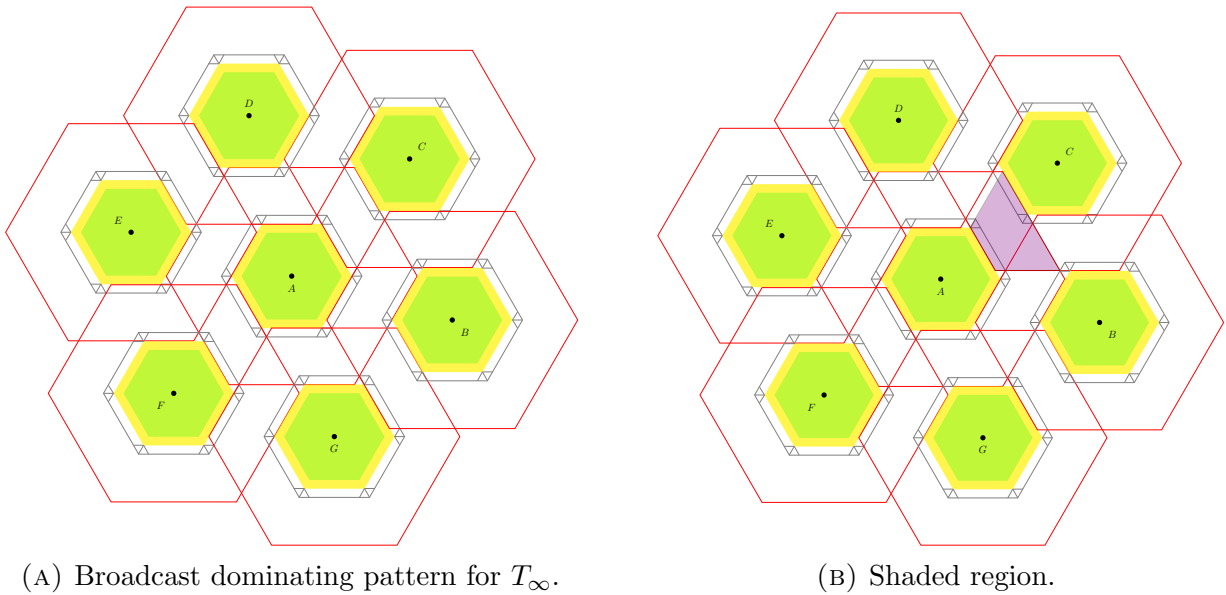
Note that by the hexagonal symmetry of the reach of the dominators, we need only prove that the vertices lying in the shaded purple region of Figure \ref{fig:shadedregion} are dominated efficiently. The remaining arguments  would follow analogously by using other dominators.

We now consider the points lying in the shaded region in two subsets as depicted in Figure~\ref{fig:dividingshadedregion}. The vertices lying on the parallelogram shaded region depicted in Figure~\ref{fig:rectangleregion} have coordinates
\begin{align}
    (t-r+1+i)\alpha_1+(1+i+j)\alpha_2 \;\mbox{where}\; 0\leq i\leq r-2 \;\mbox{and}\; 0\leq j\leq t-r
\end{align}
and receive $r-1-i$ signal from $A$, $i+1$ signal from $C$, and no signal from any other dominators. The signal received by these vertices is exactly $(r-1-i)+(i+1)=r$.
The vertices lying on the triangular shaded region depicted in Figure~\ref{fig:triangleregion} have coordinates
\begin{align}
    (t-r+2+i)\alpha_1+(1+j)\alpha_2 \;\mbox{where}\; 0\leq i\leq r-3\;\mbox{and}\; 0\leq j\leq i
\end{align}
and receive $r-2-i$ signal from $A$, $1+j$ signal from $B$, $1+i-j$ signal from $C$, and no signal from any other dominators. The  signal received by these vertices is exactly $(r-2-i)+(1+j)+(1+i-j)=r$.
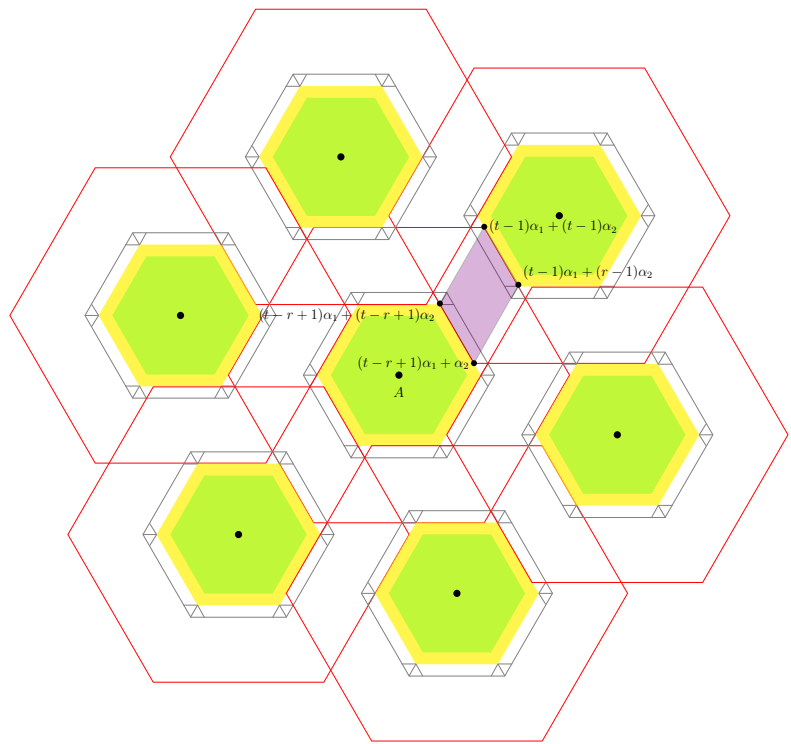
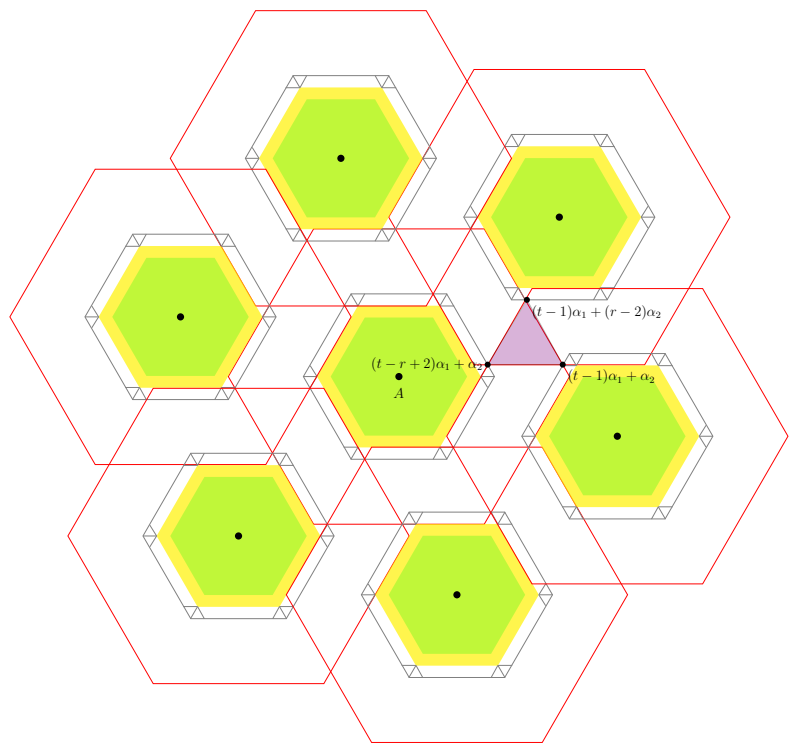
\begin{figure}
    \centering
    \begin{subfigure}[b]{0.6\textwidth}
       \resizebox{1.1\textwidth}{!}{ \tikzset{
    harrishexagon/.pic={
    \draw[thick, red] (0:2.5*\R) \foreach \x in {60,120,...,359} {
                -- (\x:2.5*\R)
            }-- cycle (90:2.5*\R);
    \draw[thick, gray] (0:1.4*\R) \foreach \x in {60,120,...,359} {
                -- (\x:1.4*\R)
            }-- cycle (90:1.4*\R);
    \fill[yellow, opacity=.7] (0:1.2*\R) \foreach \x in {60,120,...,359} {
                -- (\x:1.2*\R)
            }-- cycle (90:1.2*\R);
    \fill[green, nearly transparent] (0:\R) \foreach \x in {60,120,...,359} {
                -- (\x:\R)
            }-- cycle (90:\R);
\draw[gray](2.4,0)--(2.8,0);
\draw[gray](2.4,0)--(2.6,.32);
\draw[gray](2.4,0)--(2.6,-.32);
\draw[gray](1.2,2.07)--(1.41,2.43);
\draw[gray](1.2,2.07)--(1.62,2.07);
\draw[gray](1.2,2.07)--(1,2.43);

\draw[gray](-2.4,0)--(-2.8,0);
\draw[gray](-2.4,0)--(-2.6,.32);
\draw[gray](-2.4,0)--(-2.6,-.32);
\draw[gray](-1.2,2.07)--(-1.41,2.43);
\draw[gray](-1.2,2.07)--(-1.62,2.07);
\draw[gray](-1.2,2.07)--(-1,2.43);

\draw[gray](-1.2,-2.07)--(-1.41,-2.43);
\draw[gray](-1.2,-2.07)--(-1.62,-2.07);
\draw[gray](-1.2,-2.07)--(-1,-2.43);
\draw[gray](1.2,-2.07)--(1.41,-2.43);
\draw[gray](1.2,-2.07)--(1.62,-2.07);
\draw[gray](1.2,-2.07)--(1,-2.43);

\fill [black] (0:0) circle (3pt);
            }}
\resizebox{12cm}{!}{%

\begin{tikzpicture}[scale=0.5]
\begin{scope}[yscale=-1,xscale=1]
\pic at (0,0) {harrishexagon};
\node at (0,0) [below=2mm]{$A$};

\pic at (9.4,-9.35) {harrishexagon};
\node at (9.4,-9.35) [above=2mm]{};
\pic at (-3.4,-12.8) {harrishexagon};
\node at (-3.4,-12.8) [below=2mm]{};
\pic at (-12.8,-3.5) {harrishexagon};
\node at (-12.8,-3.5) [below=2mm]{};
\pic at (-9.4,9.35) {harrishexagon};
\node at (-9.4,9.35) [below=2mm]{};
\pic at (3.4,12.8) {harrishexagon};
\node at (3.4,12.8) [below=2mm]{};
\pic at (12.8,3.5) {harrishexagon};
\node at (12.8,3.5) [below=2mm]{};
\end{scope}
\filldraw[fill=violet!80!violet,opacity=0.3] (4.4,.7)--(7,5.3)--(5,8.7)--(2.4,4.2)--cycle;
 \fill [black] (4.4,.7) circle (5pt) node[ left] 
{{\bf{$(t-r+1)\alpha_1+\alpha_2$}}};
\fill[black] (7,5.3) circle (5pt) node [above right] {{\bf{$(t-1)\alpha_1+(r-1)\alpha_2$}}};
\fill[black] (5,8.7) circle (5pt)node [right] {{\bf{$(t-1)\alpha_1+(t-1)\alpha_2$}}};
\fill[black](2.4,4.2) circle (5pt)node [below left] {{\textbf{$(t-r+1)\alpha_1+(t-r+1)\alpha_2$}}};
\end{tikzpicture}

}}
        \caption{Parallelogram shaded region.}
        \label{fig:rectangleregion}
        \end{subfigure}
    \\
    \begin{subfigure}[b]{0.6\textwidth}
       \resizebox{1.1\textwidth}{!}{ \tikzset{
    harrishexagon/.pic={
    \draw[thick, red] (0:2.5*\R) \foreach \x in {60,120,...,359} {
                -- (\x:2.5*\R)
            }-- cycle (90:2.5*\R);
    \draw[thick, gray] (0:1.4*\R) \foreach \x in {60,120,...,359} {
                -- (\x:1.4*\R)
            }-- cycle (90:1.4*\R);
    \fill[yellow, opacity=.7] (0:1.2*\R) \foreach \x in {60,120,...,359} {
                -- (\x:1.2*\R)
            }-- cycle (90:1.2*\R);
    \fill[green, nearly transparent] (0:\R) \foreach \x in {60,120,...,359} {
                -- (\x:\R)
            }-- cycle (90:\R);
\draw[gray](2.4,0)--(2.8,0);
\draw[gray](2.4,0)--(2.6,.32);
\draw[gray](2.4,0)--(2.6,-.32);
\draw[gray](1.2,2.07)--(1.41,2.43);
\draw[gray](1.2,2.07)--(1.62,2.07);
\draw[gray](1.2,2.07)--(1,2.43);

\draw[gray](-2.4,0)--(-2.8,0);
\draw[gray](-2.4,0)--(-2.6,.32);
\draw[gray](-2.4,0)--(-2.6,-.32);
\draw[gray](-1.2,2.07)--(-1.41,2.43);
\draw[gray](-1.2,2.07)--(-1.62,2.07);
\draw[gray](-1.2,2.07)--(-1,2.43);

\draw[gray](-1.2,-2.07)--(-1.41,-2.43);
\draw[gray](-1.2,-2.07)--(-1.62,-2.07);
\draw[gray](-1.2,-2.07)--(-1,-2.43);
\draw[gray](1.2,-2.07)--(1.41,-2.43);
\draw[gray](1.2,-2.07)--(1.62,-2.07);
\draw[gray](1.2,-2.07)--(1,-2.43);

\fill [black] (0:0) circle (3pt);
            }}
\resizebox{12cm}{!}{%

\begin{tikzpicture}[scale=0.5]
\begin{scope}[yscale=-1,xscale=1]
\pic at (0,0) {harrishexagon};
\node at (0,0) [below=2mm]{$A$};

\pic at (9.4,-9.35) {harrishexagon};
\node at (9.4,-9.35) [below=2mm]{};
\pic at (-3.4,-12.8) {harrishexagon};
\node at (-3.4,-12.8) [below=2mm]{};
\pic at (-12.8,-3.5) {harrishexagon};
\node at (-12.8,-3.5) [below=2mm]{};
\pic at (-9.4,9.35) {harrishexagon};
\node at (-9.4,9.35) [below=2mm]{};
\pic at (3.4,12.8) {harrishexagon};
\node at (3.4,12.8) [below=2mm]{};
\pic at (12.8,3.5) {harrishexagon};
\node at (12.8,3.5) [below=2mm]{};
\end{scope}
\filldraw[fill=violet!80!violet,opacity=0.3] (5.2,.7)--(9.6,.7)--(7.5,4.5)--cycle;
\fill [black] (5.2,.7) circle (5pt) node[left] { $(t-r+2)\alpha_1+\alpha_2$};
\fill [black] (9.6,.7) circle (5pt) node[below right] { $(t-1)\alpha_1+\alpha_2$};
\fill [black] (7.5,4.5) circle (5pt) node[below right] {$(t-1)\alpha_1+(r-2)\alpha_2$};
\end{tikzpicture}
}}
        \caption{Triangular shaded region.}
        \label{fig:triangleregion}
        \end{subfigure}
        \caption{Subdividing the purple shaded region of Figure \ref{fig:shadedregion}.}
        \label{fig:dividingshadedregion}
\end{figure}
\end{proof}

\begin{corollary}\label{cor:mirrorpattern}
A second efficient $\tr$ broadcast domination pattern is obtained by placing dominators at every vertex of the form
\begin{align}
    [tx+(2t-r)y]\alpha_1+[(2t-r)x+(t-r)y]\alpha_2
\end{align}
\noindent {with $x,y \in \mathbb{Z}$.}
\end{corollary}

Note that the domination pattern presented in Corollary~\ref{cor:mirrorpattern} is the reflection of the pattern presented in Theorem \ref{theorem:infinitedomination}  across the line through the origin and the point $\alpha_1+\alpha_2$.  The proof of Corollary \ref{cor:mirrorpattern} is analogous top that of Theorem \ref{theorem:infinitedomination} and hence we omit the details. 
Using Theorem \ref{theorem:infinitedomination}, we provide illustrative examples of broadcast domination patters for the infinite triangular grid with $t=4$ and $1\leq r\leq 4$. These are presented in Figure \ref{fig:exampleoftheorem}.

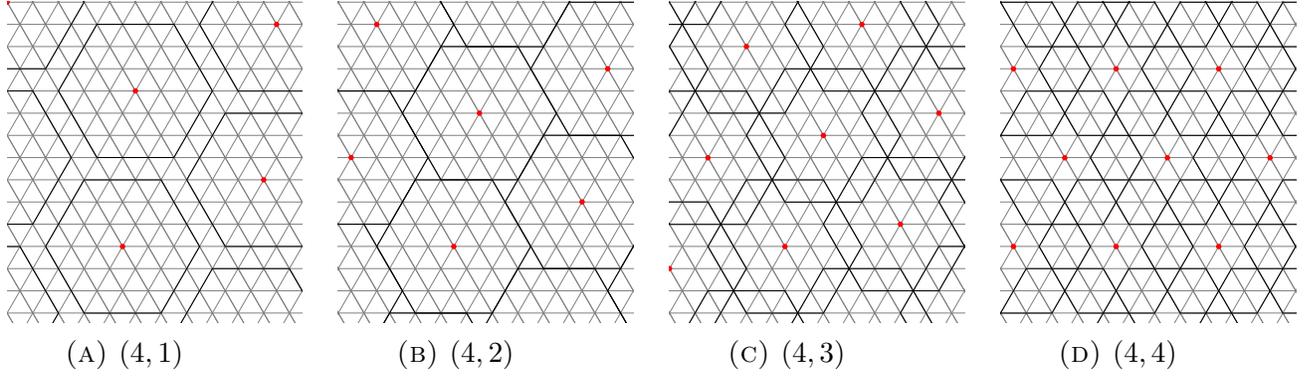
\begin{figure}[h]
    \centering
    \begin{subfigure}[t]{.2\textwidth}
    \centering
    \resizebox{1.6in}{!}{\tikzset{
    hexagon/.pic={
\draw[thick, black] (0:1.5*\R) \foreach \x in {60,120,...,359} {
                -- (\x:1.5*\R)
            }-- cycle (90:\R);
\fill [red] (0:0) circle (3pt);
            }}

\begin{tikzpicture}
\clip (1.5,.5) rectangle (13,13);

\foreach \x  in {0,...,25} {
        \draw[gray] (\x-11,0) -- ($(\x-11,0)+20*(0.5, {0.5*sqrt(3)})$);
        \draw[gray] (\x,0) -- ($(\x,0)+20*(-0.5, {0.5*sqrt(3)})$);
        \draw[gray] ($\x*(0, {0.5*sqrt(3)})$) -- ($\x*(15, {0.5*sqrt(3)})$);
        \draw[gray] (\x-5,0) -- ($(\x-5,0)+20*(-0.5, {0.5*sqrt(3)})$);
        \draw[gray] ($\x*(0, {0.5*sqrt(3)})$) -- ($\x*(-1, {0.5*sqrt(3)})$);
    }

\pic at (11,0) {hexagon};
\pic at (11.5,6.06) {hexagon};
\pic at (6,3.46) {hexagon};
\pic at (6.5,9.53) {hexagon};
\pic at (0.5,0.86) {hexagon};
\pic at (1,6.93) {hexagon};
\pic at (1.5,12.99) {hexagon};
\pic at (12,12.12) {hexagon};

\end{tikzpicture}}
    \caption{$(4,1)$}
    \end{subfigure}
    \hfill
    \begin{subfigure}[t]{.2\textwidth}
        \centering
    \resizebox{1.6in}{!}{\tikzset{
    hexagon/.pic={
\draw[thick, black] (0:1.5*\R) \foreach \x in {60,120,...,359} {
                -- (\x:1.5*\R)
            }-- cycle (90:\R);
\fill [red] (0:0) circle (3pt);
            }}

\begin{tikzpicture}
\clip (1.5,.5) rectangle (13,13);
\foreach \x  in {0,...,25} {
        \draw[gray] (\x-11,0) -- ($(\x-11,0)+20*(0.5, {0.5*sqrt(3)})$);
        \draw[gray] (\x,0) -- ($(\x,0)+20*(-0.5, {0.5*sqrt(3)})$);
        \draw[gray] ($\x*(0, {0.5*sqrt(3)})$) -- ($\x*(15, {0.5*sqrt(3)})$);
        \draw[gray] (\x-5,0) -- ($(\x-5,0)+20*(-0.5, {0.5*sqrt(3)})$);
        \draw[gray] ($\x*(0, {0.5*sqrt(3)})$) -- ($\x*(-1, {0.5*sqrt(3)})$);
    }
\pic at (5,-1.732) {hexagon};
\pic at (1,1.732) {hexagon};
\pic at (2,6.93) {hexagon};
\pic at (6,3.46) {hexagon};
\pic at (7,8.66) {hexagon};
\pic at (10,0) {hexagon};
\pic at (15,1.73) {hexagon};
\pic at (11,5.2) {hexagon};
\pic at (12,10.39) {hexagon};
\pic at (3,12.12) {hexagon};
\pic at (8,13.85) {hexagon};

\end{tikzpicture}}
    \caption{$(4,2)$}
    \end{subfigure}
    \hfill
    \begin{subfigure}[t]{.2\textwidth}
        \centering
    \resizebox{1.6in}{!}{\tikzset{
    hexagon/.pic={
\draw[thick, black] (0:1.5*\R) \foreach \x in {60,120,...,359} {
                -- (\x:1.5*\R)
            }-- cycle (90:\R);
\fill [red] (0:0) circle (3pt);
            }}

\begin{tikzpicture}
\clip (1.5,.5) rectangle (13,13);
\foreach \x  in {0,...,25} {
        \draw[gray] (\x-11,0) -- ($(\x-11,0)+20*(0.5, {0.5*sqrt(3)})$);
        \draw[gray] (\x,0) -- ($(\x,0)+20*(-0.5, {0.5*sqrt(3)})$);
        \draw[gray] ($\x*(0, {0.5*sqrt(3)})$) -- ($\x*(15, {0.5*sqrt(3)})$);
        \draw[gray] (\x-5,0) -- ($(\x-5,0)+20*(-0.5, {0.5*sqrt(3)})$);
        \draw[gray] ($\x*(0, {0.5*sqrt(3)})$) -- ($\x*(-1, {0.5*sqrt(3)})$);
    }
\pic at (4.5,-.866) {hexagon};
\pic at (1.5,2.6) {hexagon};
\pic at (3,6.92) {hexagon};
\pic at (6,3.46) {hexagon};
\pic at (7.5,7.79) {hexagon};
\pic at (9,0) {hexagon};
\pic at (13.5,.866) {hexagon};
\pic at (10.5,4.33) {hexagon};
\pic at (12,8.66) {hexagon};
\pic at (4.5,11.26) {hexagon};
\pic at (9,12.12) {hexagon};
\pic at (0,10.39) {hexagon};
\pic at (1.5,14.72) {hexagon};
\pic at (13.5,13) {hexagon};
\pic at (15,5.2) {hexagon};

\end{tikzpicture}}
    \caption{$(4,3)$}
    \end{subfigure}
    \hfill
    \begin{subfigure}[t]{.2\textwidth}
        \centering
    \resizebox{1.6in}{!}{\tikzset{
    hexagon/.pic={
\draw[thick, black] (0:1.5*\R) \foreach \x in {60,120,...,359} {
                -- (\x:1.5*\R)
            }-- cycle (90:\R);
\fill [red] (0:0) circle (3pt);
            }}

\begin{tikzpicture}
\clip (1.5,.5) rectangle (13,13);
\foreach \x  in {0,...,25} {
        \draw[gray] (\x-11,0) -- ($(\x-11,0)+20*(0.5, {0.5*sqrt(3)})$);
        \draw[gray] (\x,0) -- ($(\x,0)+20*(-0.5, {0.5*sqrt(3)})$);
        \draw[gray] ($\x*(0, {0.5*sqrt(3)})$) -- ($\x*(15, {0.5*sqrt(3)})$);
        \draw[gray] (\x-5,0) -- ($(\x-5,0)+20*(-0.5, {0.5*sqrt(3)})$);
        \draw[gray] ($\x*(0, {0.5*sqrt(3)})$) -- ($\x*(-1, {0.5*sqrt(3)})$);
    }
\pic at (4,0) {hexagon};
\pic at (8,0) {hexagon};
\pic at (12,0) {hexagon};

\pic at (2,3.46) {hexagon};
\pic at (6,3.46) {hexagon};
\pic at (10,3.46) {hexagon};
\pic at (14,3.46) {hexagon};

\pic at (4,6.92) {hexagon};
\pic at (8,6.92) {hexagon};
\pic at (12,6.92) {hexagon};

\pic at (2,10.39) {hexagon};
\pic at (6,10.39) {hexagon};
\pic at (10,10.39) {hexagon};
\pic at (14,10.39) {hexagon};

\pic at (4,13.85) {hexagon};
\pic at (8,13.85) {hexagon};
\pic at (12,13.85) {hexagon};
\end{tikzpicture}}
    \caption{$(4,4)$}
    \end{subfigure}
\caption{Efficient broadcasts on $\T$.}
    \label{fig:exampleoftheorem}
\end{figure}

\subsection{Upper bounds for broadcast domination numbers}\label{sec:newbounds}
In this section we use Theorem~\ref{theorem:infinitedomination} to establish upper bounds for the $(t,r)$ broadcast domination of triangular matchstick graphs when $(t,r)\in\{(2,1),(3,1),(3,2),(4,1),(4,2),(4,3),(t,t)\}$.

\begin{theorem}\label{lemma:newupperbound21}
If $n=7k+\beta$ with $0\leq \beta\leq 6$, then
\[\gamma_{2,1}(T_n)\leq \begin{cases}
3k(k+2)&\mbox{if $\beta=0$}\\
3(k+1)(k+3)&\mbox{if $\beta\neq 0$}.
\end{cases}\]
\end{theorem}
\begin{proof}
The method of proof is to embed the matchstick graph $T_n$ with its left bottom corner at the origin. Doing so we then place dominators in $\T$ using Theorem \ref{theorem:infinitedomination} with $t=2$ and $r=1$. 
Note that by construction the origin is a dominator in $T_n$.
By Equation \eqref{thm:equation}, we can compute that in the $(2,1)$ broadcast domination pattern, the  dominator in $T_n$ which is closest to the origin and has coordinates $m\alpha_1$ is given when $m=7$. 
Hence, we can subdivide $T_n$ into $T_{7}$'s, by tiling it as depicted in Figure \ref{fig:subdivision}. Let $n=7k+\beta$ with $0\leq \beta\leq 6$. If  $\beta=0$, then the base of $T_n$ subdivides evenly into $k$ $T_7$'s, if not then we add an additional row of $T_7$'s as depicted in Figure \ref{fig:21subdivisionwithremainder}. 
If $\beta=0$, then the number of $T_7$'s used in the subdivision of $T_n$ is $k^2$, and if not it is $(k+1)^2$. We also note that every $T_7$ can be dominated by at most $9$ dominators, see Figure \ref{fig:upperbound21}. Hence, we have shown that 
\[
\gamma_{2,1}(T_n)\leq \begin{cases}
9k^2&\mbox{if $\beta=0$}\\
9(k+1)^2&\mbox{if $\beta\neq 0$}.
\end{cases}
\]

\begin{figure}[h]
    \centering
    \begin{subfigure}{.4\textwidth}
    \resizebox{\textwidth}{!}{
    \tikzset{
    hex/.pic={
\coordinate (0;0) at (0,0); 
\foreach \c in {1,...,2}{%
\foreach \i in {0,...,5}{%
\pgfmathtruncatemacro\j{\c*\i}
\coordinate (\c;\j) at (60*\i:\c);  
} }
\foreach \i in {0,2,...,10}{%
\pgfmathtruncatemacro\j{mod(\i+2,12)} 
\pgfmathtruncatemacro\k{\i+1}
\coordinate (2;\k) at ($(2;\i)!.5!(2;\j)$) ;}

 \foreach \i in {0,...,6}{%
 \pgfmathtruncatemacro\k{\i}
 \pgfmathtruncatemacro\l{10-\i}
 \draw[gray] (2;\k)--(2;\l); 
 \pgfmathtruncatemacro\k{6-\i} 
 \pgfmathtruncatemacro\l{mod(8+\i,12)}   
 \draw[gray] (2;\k)--(2;\l); 
 \pgfmathtruncatemacro\k{9-\i} 
 \pgfmathtruncatemacro\l{mod(9+\i,12)}   
 \draw[gray] (2;\k)--(2;\l);} 
 
\fill [black] (0;0) circle (3.5pt);

\draw[line width=.75mm, violet] (0:\R) \foreach \x in {60,120,...,359} {
                -- (\x:\R)
            }-- cycle (90:\R);

}}
\begin{tikzpicture}[scale=0.50]
    \foreach \row in {0, 1,...,14} {
        \draw[gray, thick] ($\row*(0.5, {0.5*sqrt(3)})$) -- ($(14,0)+\row*(-0.5, {0.5*sqrt(3)})$);
        \draw[gray, thick] ($\row*(1, 0)$) -- ($(14/2,{14/2*sqrt(3)})+\row*(0.5,{-0.5*sqrt(3)})$);
        \draw[gray, thick] ($\row*(1, 0)$) -- ($(0,0)+\row*(0.5,{0.5*sqrt(3)})$);
    }
    
\draw[ultra thick, black] (.05,.05)--(13.95,.05);   
\draw[ultra thick, black] (.08,.05)--(7.05, 12.08);   
\draw[ultra thick, black] (13.95,.05)--(7.05, 12.08);    
\draw[ultra thick, black] (6.95,.05)--(3.5, 6.06218);
\draw[ultra thick, black] (10.5, 6.06218)--(3.5, 6.06218);
\draw[ultra thick, black] (7,.05)--(10.5, 6.06218);

   \foreach \row in {0, 1,...,14} {
        \draw[gray, thick] ($(10.5, 18.18)+\row*(0.5, {0.5*sqrt(3)})$) -- ($(24.5, 18.1865)+\row*(-0.5, {0.5*sqrt(3)})$);
             \draw[gray, thick] ($(21,0)+\row*(0.5, 0.866025)$) -- ($(10.5, 18.1865)+\row*(0.5,{0.5*sqrt(3)})$);
    }

    \foreach \row in {0,1,...,21} {
        \draw[gray, thick] ($(21,0)+\row*(-0.5, {0.5*sqrt(3)})$) -- ($(35,0)+\row*(-0.5, {0.5*sqrt(3)})$);
        \draw[gray, thick] ($(10.5, 18.1865)-\row*(-0.5, {0.5*sqrt(3)})$) -- ($(17.5, 30.3109)-\row*(-0.5, {0.5*sqrt(3)})$);
    }

    \foreach \row in {0, 1,...,14} {
        \draw[gray, thick] ($(21,0)+\row*(0.5, {0.5*sqrt(3)})$) -- ($(35,0)+\row*(-0.5, {0.5*sqrt(3)})$);
        \draw[gray, thick] ($(21,0)+\row*(1, 0)$) -- ($(28., 12.1244)+\row*(0.5,{-0.5*sqrt(3)})$);
        \draw[gray, thick] ($(21,0)+\row*(1, 0)$) -- ($(21,0)+\row*(0.5,{0.5*sqrt(3)})$);
    }

    \draw[line width=.75mm,white](17., 6.9282)--(24., 19.0526);
    \draw[line width=1.75mm,white](16.5, 7.79423)--(23.5, 19.9186);
    \draw[line width=1.75mm,white]( 16., 8.66025)--(23., 20.7846);   
    \draw[line width=1.75mm,white](16.5, 7.79423)--(17.5, 7.79423);   
    \foreach \row in {0, 1,...,14} {
    \draw[line width=1.75mm,white]($(17., 6.9282)+\row*(0.5, 0.866025)$) -- ($(16., 8.66025)+\row*(0.5, 0.866025)$);    
    }
    \foreach \row in {0, 1,...,21} {
    \draw[line width=1.75mm,white]($(16., 8.66025)+\row*(0.5, 0.866025)$) -- ($(18., 8.66025)+\row*(0.5, 0.866025)$);    
    }


\draw[ultra thick, black] (.05,.05)--(13.95,.05);   
\draw[ultra thick, black] (.08,.05)--(7.05, 12.08);   
\draw[ultra thick, black] (13.95,.05)--(7.05, 12.08);    
\draw[ultra thick, black] (6.95,.05)--(3.5, 6.06218);
\draw[ultra thick, black] (10.5, 6.06218)--(3.5, 6.06218);
\draw[ultra thick, black] (7,.05)--(10.5, 6.06218);

\draw[ultra thick, black] (24., 19.0526)--(35., 0.);
\draw[ultra thick, black] (17.5, 30.3109)--(23., 20.7846);
\draw[ultra thick, black] (21.05,.05)--(35,.05);   
\draw[ultra thick, black] (28,0)--(31.5, 6.06218);
\draw[ultra thick, black] (21,0)--(28., 12.1244);
\draw[ultra thick, black] (28,0)--(20.5, 12.9904);
\draw[ultra thick, black] (21,0)--(17., 6.9282);
\draw[ultra thick, black] (17.5, 6.06218)--(31.5, 6.06218);
\draw[ultra thick, black] (24.5, 18.1865)--(17.5, 6.06218);
\draw[ultra thick, black] (21., 12.1244)--(28., 12.1244);
\draw[ultra thick, black] (17.5, 30.3109)--(10.5, 18.1);
\draw[ultra thick, black] (10.5, 18.1865)--(16., 8.66025);
\draw[ultra thick, black] (21., 24.2487)--(14., 24.2487);
\draw[ultra thick, black] (14., 24.2487)--(19.5, 14.7224);
\draw[ultra thick, black] (21., 24.2487)--(14., 12.1244);
\draw[ultra thick, black] (10.5, 18.1865)--(21.5, 18.1865);
\draw[ultra thick, black] (18., 12.1244)--(14., 12.1244);

\draw[line width=1.2mm, red] (0,0)--(15,0);
\node[red] at (17,0) {{\bf\Huge$\ldots$}};
\draw[line width=1.2mm, red] (19,0)--(35,0);
\draw[line width=1.2mm, red] (0,0)--(7.5,13);
\node[red, rotate=60] at (8.5, 14.7224) {{\bf\Huge$\ldots$}};
\draw[line width=1.2mm, red] (9.5,16.44)--(17.5, 30.3109)--(23., 20.7846);
\node[red, rotate=-60] at (23.5, 19.9186) {{\bf\Huge$\ldots$}};
\draw[line width=1.2mm, red] (35,0)--(24., 19.0526);
\node[black] at (0,10) {{\bf\Huge$T_n$}};
\draw [decorate,decoration={brace, mirror, amplitude=10pt}, ultra thick, black]
(0,-1) -- (35,-1) node [black,midway,yshift=-.8cm] 
{\Large$7k$};
\end{tikzpicture}
}
    \caption{$T_n$ with $n=7k$}\label{fig:21perfectsubdivision}
    \end{subfigure}
    \hspace{5mm}
    \begin{subfigure}{.4\textwidth}
    \resizebox{\textwidth}{!}{
    \tikzset{
    hex/.pic={
\coordinate (0;0) at (0,0); 
\foreach \c in {1,...,2}{%
\foreach \i in {0,...,5}{%
\pgfmathtruncatemacro\j{\c*\i}
\coordinate (\c;\j) at (60*\i:\c);  
} }
\foreach \i in {0,2,...,10}{%
\pgfmathtruncatemacro\j{mod(\i+2,12)} 
\pgfmathtruncatemacro\k{\i+1}
\coordinate (2;\k) at ($(2;\i)!.5!(2;\j)$) ;}

 \foreach \i in {0,...,6}{%
 \pgfmathtruncatemacro\k{\i}
 \pgfmathtruncatemacro\l{10-\i}
 \draw[gray] (2;\k)--(2;\l); 
 \pgfmathtruncatemacro\k{6-\i} 
 \pgfmathtruncatemacro\l{mod(8+\i,12)}   
 \draw[gray] (2;\k)--(2;\l); 
 \pgfmathtruncatemacro\k{9-\i} 
 \pgfmathtruncatemacro\l{mod(9+\i,12)}   
 \draw[gray] (2;\k)--(2;\l);} 
 
\fill [black] (0;0) circle (3.5pt);

\draw[line width=.75mm, violet] (0:\R) \foreach \x in {60,120,...,359} {
                -- (\x:\R)
            }-- cycle (90:\R);

}}

\begin{tikzpicture}[scale=0.50]
    \foreach \row in {0, 1,...,14} {
        \draw[gray, thick] ($\row*(0.5, {0.5*sqrt(3)})$) -- ($(14,0)+\row*(-0.5, {0.5*sqrt(3)})$);
        \draw[gray, thick] ($\row*(1, 0)$) -- ($(14/2,{14/2*sqrt(3)})+\row*(0.5,{-0.5*sqrt(3)})$);
        \draw[gray, thick] ($\row*(1, 0)$) -- ($(0,0)+\row*(0.5,{0.5*sqrt(3)})$);
    }
    
\draw[ultra thick, black] (.05,.05)--(13.95,.05);   
\draw[ultra thick, black] (.08,.05)--(7.05, 12.08);   
\draw[ultra thick, black] (13.95,.05)--(7.05, 12.08);    
\draw[ultra thick, black] (6.95,.05)--(3.5, 6.06218);
\draw[ultra thick, black] (10.5, 6.06218)--(3.5, 6.06218);
\draw[ultra thick, black] (7,.05)--(10.5, 6.06218);

   \foreach \row in {0, 1,...,14} {
        \draw[gray, thick] ($(10.5, 18.18)+\row*(0.5, {0.5*sqrt(3)})$) -- ($(24.5, 18.1865)+\row*(-0.5, {0.5*sqrt(3)})$);
             \draw[gray, thick] ($(21,0)+\row*(0.5, 0.866025)$) -- ($(10.5, 18.1865)+\row*(0.5,{0.5*sqrt(3)})$);
    }

    \foreach \row in {0,1,...,21} {
        \draw[gray, thick] ($(21,0)+\row*(-0.5, {0.5*sqrt(3)})$) -- ($(35,0)+\row*(-0.5, {0.5*sqrt(3)})$);
        \draw[gray, thick] ($(10.5, 18.1865)-\row*(-0.5, {0.5*sqrt(3)})$) -- ($(17.5, 30.3109)-\row*(-0.5, {0.5*sqrt(3)})$);
    }

    \foreach \row in {0, 1,...,14} {
        \draw[gray, thick] ($(21,0)+\row*(0.5, {0.5*sqrt(3)})$) -- ($(35,0)+\row*(-0.5, {0.5*sqrt(3)})$);
        \draw[gray, thick] ($(21,0)+\row*(1, 0)$) -- ($(28., 12.1244)+\row*(0.5,{-0.5*sqrt(3)})$);
        \draw[gray, thick] ($(21,0)+\row*(1, 0)$) -- ($(21,0)+\row*(0.5,{0.5*sqrt(3)})$);
    }

    \draw[line width=.75mm,white](17., 6.9282)--(24., 19.0526);
    \draw[line width=1.75mm,white](16.5, 7.79423)--(23.5, 19.9186);
    \draw[line width=1.75mm,white]( 16., 8.66025)--(23., 20.7846);   
    \draw[line width=1.75mm,white](16.5, 7.79423)--(17.5, 7.79423);   
    \foreach \row in {0, 1,...,14} {
    \draw[line width=1.75mm,white]($(17., 6.9282)+\row*(0.5, 0.866025)$) -- ($(16., 8.66025)+\row*(0.5, 0.866025)$);    
    }
    \foreach \row in {0, 1,...,21} {
    \draw[line width=1.75mm,white]($(16., 8.66025)+\row*(0.5, 0.866025)$) -- ($(18., 8.66025)+\row*(0.5, 0.866025)$);    
    }


\draw[ultra thick, black] (.05,.05)--(13.95,.05);   
\draw[ultra thick, black] (.08,.05)--(7.05, 12.08);   
\draw[ultra thick, black] (13.95,.05)--(7.05, 12.08);    
\draw[ultra thick, black] (6.95,.05)--(3.5, 6.06218);
\draw[ultra thick, black] (10.5, 6.06218)--(3.5, 6.06218);
\draw[ultra thick, black] (7,.05)--(10.5, 6.06218);

\draw[ultra thick, black] (24., 19.0526)--(35., 0.);
\draw[ultra thick, black] (17.5, 30.3109)--(23., 20.7846);
\draw[ultra thick, black] (21.05,.05)--(35,.05);   
\draw[ultra thick, black] (28,0)--(31.5, 6.06218);
\draw[ultra thick, black] (21,0)--(28., 12.1244);
\draw[ultra thick, black] (28,0)--(20.5, 12.9904);
\draw[ultra thick, black] (21,0)--(17., 6.9282);
\draw[ultra thick, black] (17.5, 6.06218)--(31.5, 6.06218);
\draw[ultra thick, black] (24.5, 18.1865)--(17.5, 6.06218);
\draw[ultra thick, black] (21., 12.1244)--(28., 12.1244);
\draw[ultra thick, black] (17.5, 30.3109)--(10.5, 18.1);
\draw[ultra thick, black] (10.5, 18.1865)--(16., 8.66025);
\draw[ultra thick, black] (21., 24.2487)--(14., 24.2487);
\draw[ultra thick, black] (14., 24.2487)--(19.5, 14.7224);
\draw[ultra thick, black] (21., 24.2487)--(14., 12.1244);
\draw[ultra thick, black] (10.5, 18.1865)--(21.5, 18.1865);
\draw[ultra thick, black] (18., 12.1244)--(14., 12.1244);

\draw[line width=1.2mm, red] (0,0)--(15,0);
\node[red] at (17,0) {{\bf\Huge$\ldots$}};
\draw[line width=1.2mm, red] (19,0)--(33,0);
\draw[line width=1.2mm, red] (0,0)--(7.5,13);
\node[red, rotate=60] at (8.5, 14.7224) {{\bf\Huge$\ldots$}};
\draw[line width=1.2mm, red] (9.5,16.44)--(16.5, 28.5788)--(22., 19.0526);
\node[red, rotate=-60] at (22.5, 18.1865) {{\bf\Huge$\ldots$}};
\draw[line width=1.2mm, red] (33,0)--(23., 17.3205);
\node[black] at (0,10) {{\bf\Huge$T_n$}};
\draw [decorate,decoration={brace, mirror, amplitude=10pt}, ultra thick, black]
(0,-1) -- (28,-1) node [black,midway,yshift=-.8cm] 
{\Large$7k$};
\draw [decorate,decoration={brace, mirror, amplitude=10pt}, ultra thick, black]
(28.2,-1) -- (33,-1) node [black,midway,yshift=-.8cm] 
{\Large$\beta$};
\end{tikzpicture}
}
    \caption{$T_n$ with $n=7k+\beta$}\label{fig:21subdivisionwithremainder}
    \end{subfigure}
    \hfill
    \begin{subfigure}{.2\textwidth}
    \resizebox{0.95\textwidth}{!}{
    \input{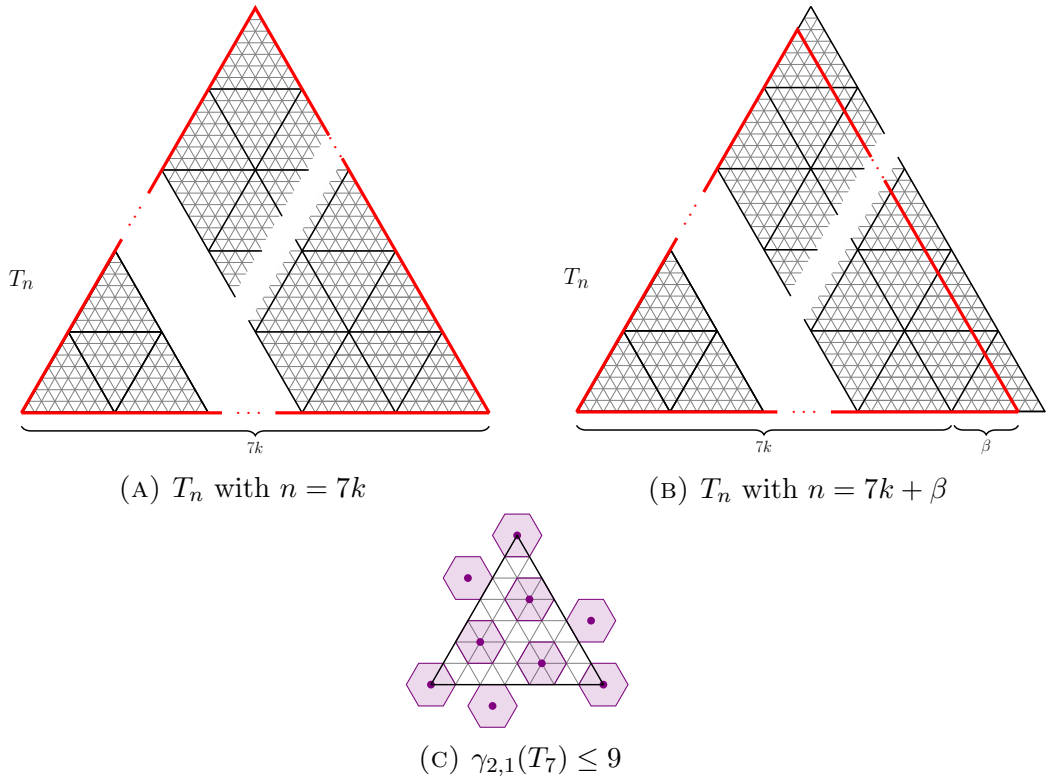}}
    \caption{$\gamma_{2,1}(T_7)\leq 9$}\label{fig:upperbound21}
    \end{subfigure}
\caption{Subdividing $T_n$ into $T_7$'s.}
\label{fig:subdivision}
\end{figure}

However, we are currently double counting the dominators used (per $T_7$) whenever two $T_7$'s share an edge in the subdivision of $T_n$. 
By Lemma \ref{lem:interioredgesinTn} we know that the number of interior edges in $T_\ell$ is given by $3\Delta_{\ell-1}=3\binom{\ell}{2}$.
Since the number of interior edges in the subdivision of $T_n$ into $T_7$'s is the same as the number of interior edges in the graph $T_{k}$ (when $\beta=0$) or $T_{k+1}$ (when $\beta\neq 0$) we can subtract the double count per each of the interior edges by 
noting that 4 dominators are used when dominating any edge of $T_7$. Thus, we have established that 
\[\gamma_{2,1}(T_n)\leq \begin{cases}
9k^2-12\Delta_{k-1}&\mbox{if $\beta=0$}\\
9(k+1)^2-12\Delta_k&\mbox{if $\beta\neq 0$}
\end{cases}\]
from which the result follows.
\end{proof}

\begin{theorem}\label{lemma:newupperbound31}
If $n=19k+\beta$ with $0\leq \beta\leq 18$, then
\[\gamma_{3,1}(T_n)\leq \begin{cases}
9k(k+1)&\mbox{if $\beta=0$}\\
9(k+1)(k+2)&\mbox{if $\beta\neq 0$}.
\end{cases}\]
\end{theorem}
\begin{proof}
We proceed as in the previous proof with $t=3$ and $r=1$. 
Note that by construction the origin is a dominator in $T_n$ and by Equation \eqref{thm:equation}, we can compute that in the $(3,1)$ broadcast domination pattern, the  dominator in $T_n$ which is closest to the origin and has coordinates $m\alpha_1$ is given when $m=19$. 
Hence, we can subdivide $T_n$ into $T_{19}$'s, as before. Let $n=19k+\beta$ with $0\leq \beta\leq 18$. If  $\beta=0$, then the base of $T_n$ subdivides evenly into $k$ $T_{19}$'s, if not then we add an additional row of $T_{19}$'s. 
If $\beta=0$, then the number of $T_{19}$'s used in the subdivision of $T_n$ is $k^2$, and if not it is $(k+1)^2$. We also note that every $T_{19}$ can be dominated by at most $18$ dominators, see Figure \ref{fig:upperbound31}. Hence, we have shown that 
\[
\gamma_{3,1}(T_n)\leq \begin{cases}
18k^2&\mbox{if $\beta=0$}\\
18(k+1)^2&\mbox{if $\beta\neq 0$}.
\end{cases}
\]

\begin{figure}[h]
    \centering
    \input{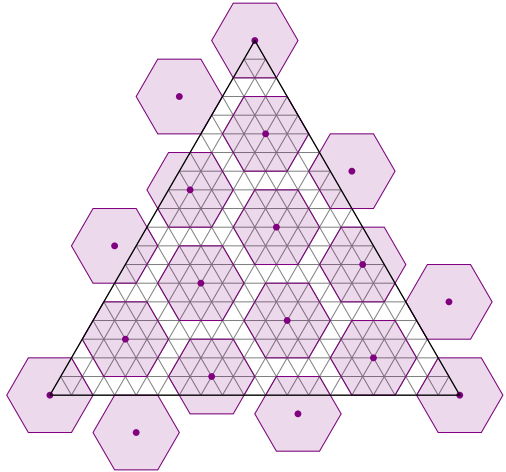}
    \caption{Dominating $T_{19}$ when $t=3$ and $r=1$.}
    \label{fig:upperbound31}
\end{figure}

However, we are currently double counting the dominators used (per $T_{19}$) whenever two $T_{19}$'s share an edge in the subdivision of $T_n$. Recall that the number of interior edges in $T_\ell$ is given by $3\Delta_{\ell-1}$. 
Since the number of interior edges in the subdivision of $T_n$ into $T_{19}$'s is the same as the number of interior edges in the graph $T_{k}$ (when $\beta=0$) or $T_{k+1}$ (when $\beta\neq 0$) we can subtract the double count per each of the interior edges by 
noting that $6$ dominators are used when dominating any edge of $T_{19}$. Thus, we have established that 
\[\gamma_{3,1}(T_n)\leq \begin{cases}
18k^2-18\Delta_{k-1}&\mbox{if $\beta=0$}\\
18(k+1)^2-18\Delta_k&\mbox{if $\beta\neq 0$}
\end{cases}\]
from which the result follows.
\end{proof}
Note that the upper bound presented in Theorem~\ref{lemma:newupperbound31} is better than that of Theorem~\ref{thm:NEWdomnumberfort=3} as $\Delta_{\left\lfloor\frac{n}{4}+1\right\rfloor}$ grows like $ \frac{n^2}{32}$, while the bound in Theorem~\ref{lemma:newupperbound31} grows like $\frac{9n^2}{361}$, which is smaller.

Using analogous arguments, as those presented in the proofs of Theorems \ref{lemma:newupperbound21} and \ref{lemma:newupperbound31}, we establish upper bounds for the $(t,r)$ broadcast domination of $T_n$, when \[(t,r)\in\{(3,2),(4,1),(4,2),(4,3),(t,t)\}.\] We consolidate these results in Table~\ref{tableofbounds2}.

    \begin{longtable}{|c|m{1in}|m{1.7in}|c|}\hline
$(t,r)$         &$n=\ell k+\beta$, $0\leq\beta\leq k-1$  &Dominating $T_\ell$&Upper bound for $\gamma_{t,r}(T_n)$   \\\hline\hline
         $(2,1)$& $n=7k+\beta$, $0\leq \beta\leq 6$
         &\resizebox{1.6in}{!}{\tikzset{
    hexagon/.pic={
\coordinate (0;0) at (0,0); 
\foreach \c in {1,...,2}{%
\foreach \i in {0,...,5}{%
\pgfmathtruncatemacro\j{\c*\i}
\coordinate (\c;\j) at (60*\i:\c);  
} }

%

 
\fill [violet] (0;0) circle (4.5pt);
\newdimen\R
\R=1 cm
\filldraw[fill=violet!80!violet,opacity=0.15] (0:\R) \foreach \x in {60,120,...,359} {
                -- (\x:\R)
            }-- cycle (90:\R);
\draw[thick, violet] (0:\R) \foreach \x in {60,120,...,359} {
                -- (\x:\R)
            }-- cycle (90:\R);
            }}
\begin{tikzpicture}
    \foreach \row in {0, 1,...,7} {
        \draw[gray, thick] ($\row*(0.5, {0.5*sqrt(3)})$) -- ($(7,0)+\row*(-0.5, {0.5*sqrt(3)})$);
        \draw[gray, thick] ($\row*(1, 0)$) -- ($(7/2,{7/2*sqrt(3)})+\row*(0.5,{-0.5*sqrt(3)})$);
        \draw[gray, thick] ($\row*(1, 0)$) -- ($(0,0)+\row*(0.5,{0.5*sqrt(3)})$);
    }
\pic at (0,0)  {hexagon};
\pic at (7,0)  {hexagon};
\pic at (2.5, -0.866025) {hexagon};
\pic at (4.5, 0.866025) {hexagon};
\pic at (6.5, 2.59808) {hexagon};
\pic at (2., 1.73205) {hexagon};
\pic at (4., 3.4641) {hexagon};
\pic at (1.5, 4.33013) {hexagon};
\pic at (3.5, 6.06218) {hexagon};
\draw[black, ultra thick](0,0)--(7,0)--(3.5, 6.06218)--(0,0);
\end{tikzpicture}
}
         &$\gamma_{2,1}(T_n)\leq \begin{cases}
3k(k+2)&\mbox{if $\beta=0$}\\
3(k+1)(k+3)&\mbox{if $\beta\neq 0$}.
\end{cases}$\\\hline
         $(3,1)$&$n=19k+\beta$, $0\leq \beta\leq 18$&\resizebox{1.6in}{!}{\tikzset{
    hexagon/.pic={
\coordinate (0;0) at (0,0); 
\foreach \c in {1,...,2}{%
\foreach \i in {0,...,5}{%
\pgfmathtruncatemacro\j{\c*\i}
\coordinate (\c;\j) at (60*\i:\c);  
} }


 
\fill [violet] (0;0) circle (4.5pt);
\newdimen\R
\R=2 cm
\filldraw[fill=violet!80!violet,opacity=0.15] (0:\R) \foreach \x in {60,120,...,359} {
                -- (\x:\R)
            }-- cycle (90:\R);
\draw[thick, violet] (0:\R) \foreach \x in {60,120,...,359} {
                -- (\x:\R)
            }-- cycle (90:\R);
            }}
\resizebox{.4\textwidth}{!}{%
 \begin{tikzpicture}
    \foreach \row in {0, 1,...,19} {
        \draw[gray, thick] ($\row*(0.5, {0.5*sqrt(3)})$) -- ($(19,0)+\row*(-0.5, {0.5*sqrt(3)})$);
        \draw[gray, thick] ($\row*(1, 0)$) -- ($(19/2,{19/2*sqrt(3)})+\row*(0.5,{-0.5*sqrt(3)})$);
        \draw[gray, thick] ($\row*(1, 0)$) -- ($(0,0)+\row*(0.5,{0.5*sqrt(3)})$);
    }
\pic at (0,0)  {hexagon};
\pic at (19,0)  {hexagon};
\pic at (9.5, 16.4545)  {hexagon};
\pic at (4., -1.73205)  {hexagon};
\pic at (3.5, 2.59808)  {hexagon};
\pic at (7.5, 0.866025)  {hexagon};
\pic at (11.5, -0.866025)  {hexagon};
\pic at (15., 1.73205)  {hexagon};
\pic at (3., 6.9282)  {hexagon};
\pic at (7., 5.19615)  {hexagon};
\pic at (11., 3.4641)  {hexagon};
\pic at (6.5, 9.52628)  {hexagon};
\pic at (10.5, 7.79423)  {hexagon};
\pic at (14.5, 6.06218)  {hexagon};
\pic at (18.5, 4.33013)  {hexagon};
\pic at (6., 13.8564)  {hexagon};
\pic at (10., 12.1244)  {hexagon};
\pic at (14., 10.3923)  {hexagon};
\draw[black, ultra thick](0,0)--(19,0)--(9.5, 16.4545)--(0,0);
\end{tikzpicture}}} &$\gamma_{3,1}(T_n)\leq\begin{cases}
9k(k+1)&\mbox{if $\beta=0$}\\
9(k+1)(k+2)&\mbox{if $\beta\neq 0$}.
\end{cases}$\\\hline
         $(3,2)$&$n=13k+\beta$, $0\leq \beta\leq 12$ &\resizebox{1.6in}{!}{

\tikzset{
    hexagon/.pic={
\coordinate (0;0) at (0,0); 
\foreach \c in {1,...,2}{%
\foreach \i in {0,...,5}{%
\pgfmathtruncatemacro\j{\c*\i}
\coordinate (\c;\j) at (60*\i:\c);  
} }


 
\fill [violet] (0;0) circle (4.5pt);
\newdimen\R
\R=2 cm
\filldraw[fill=violet!80!violet,opacity=0.15] (0:\R) \foreach \x in {60,120,...,359} {
                -- (\x:\R)
            }-- cycle (90:\R);
\draw[thick, violet] (0:\R) \foreach \x in {60,120,...,359} {
                -- (\x:\R)
            }-- cycle (90:\R);
            }}
\resizebox{1\textwidth}{!}{%
 \begin{tikzpicture}
    \foreach \row in {0, 1,...,13} {
        \draw[gray, thick] ($\row*(0.5, {0.5*sqrt(3)})$) -- ($(13,0)+\row*(-0.5, {0.5*sqrt(3)})$);
        \draw[gray, thick] ($\row*(1, 0)$) -- ($(13/2,{13/2*sqrt(3)})+\row*(0.5,{-0.5*sqrt(3)})$);
        \draw[gray, thick] ($\row*(1, 0)$) -- ($(0,0)+\row*(0.5,{0.5*sqrt(3)})$);
    }
\pic at (0,0)  {hexagon};
\pic at (13,0)  {hexagon};
\pic at (3.5, -0.866025)  {hexagon};
\pic at (7., -1.73205)  {hexagon};
\pic at (2.5, 2.59808)  {hexagon};
\pic at (6., 1.73205)  {hexagon};
\pic at (9.5, 0.866025)  {hexagon};
\pic at (5., 5.19615)  {hexagon};
\pic at (8.5, 4.33013)  {hexagon};
\pic at (12., 3.4641)  {hexagon};
\pic at (1.5, 6.06218)  {hexagon};
\pic at (4., 8.66025)  {hexagon};
\pic at (7.5, 7.79423)  {hexagon};
\pic at (11., 6.9282)  {hexagon};
\pic at (6.5, 11.2583)  {hexagon};
\draw[black, ultra thick] (0, 0)-- (13,0)-- (6.5, 11.2583)-- (0, 0);
\end{tikzpicture}}}&$\gamma_{3,2}(T_n)\leq\begin{cases}
         3k(2k+3)&\mbox{if $\beta=0$}\\3(k+1)(2k+5)&\mbox{if $\beta\neq0$}\end{cases}$\\\hline
          $(4,1)$&$n=37k+\beta$, $0\leq \beta\leq 36$ &\resizebox{1.6in}{!}{\tikzset{
    hexagon/.pic={
\coordinate (0;0) at (0,0); 
\foreach \c in {1,...,2}{%
\foreach \i in {0,...,5}{%
\pgfmathtruncatemacro\j{\c*\i}
\coordinate (\c;\j) at (60*\i:\c);  
} }

%
%
 
\fill [violet] (0;0) circle (4.5pt);
\newdimen\R
\R=3 cm
\filldraw[fill=violet!80!violet,opacity=0.15] (0:\R) \foreach \x in {60,120,...,359} {
                -- (\x:\R)
            }-- cycle (90:\R);
\draw[thick, violet] (0:\R) \foreach \x in {60,120,...,359} {
                -- (\x:\R)
            }-- cycle (90:\R);
            }}
\resizebox{1\textwidth}{!}{%
 \begin{tikzpicture}
    \foreach \row in {0, 1,...,37} {
        \draw[gray, thick] ($\row*(0.5, {0.5*sqrt(3)})$) -- ($(37,0)+\row*(-0.5, {0.5*sqrt(3)})$);
        \draw[gray, thick] ($\row*(1, 0)$) -- ($(37/2,{37/2*sqrt(3)})+\row*(0.5,{-0.5*sqrt(3)})$);
        \draw[gray, thick] ($\row*(1, 0)$) -- ($(0,0)+\row*(0.5,{0.5*sqrt(3)})$);
    }
\pic at (0,0)  {hexagon};
\pic at (37,0)  {hexagon};
\pic at (5.5, -2.59808)  {hexagon};
\pic at (5., 3.4641)  {hexagon};
\pic at (10.5, 0.866025)  {hexagon};
\pic at (16., -1.73205)  {hexagon};
\pic at (4.5, 9.52628)  {hexagon};
\pic at (10., 6.9282)  {hexagon};
\pic at (15.5, 4.33013)  {hexagon};
\pic at (21., 1.73205)  {hexagon};
\pic at (26.5, -0.866025)  {hexagon};
\pic at (31.5, 2.59808)  {hexagon};
\pic at (26., 5.19615)  {hexagon};
\pic at (20.5, 7.79423)  {hexagon};
\pic at (15., 10.3923)  {hexagon};
\pic at (9.5, 12.9904)  {hexagon};
\pic at (9., 19.0526)  {hexagon};
\pic at (14.5, 16.4545)  {hexagon};
\pic at (20., 13.8564)  {hexagon};
\pic at (25.5, 11.2583)  {hexagon};
\pic at (31., 8.66025)  {hexagon};
\pic at (36.5, 6.06218)  {hexagon};
\pic at (30.5, 14.7224)  {hexagon};
\pic at (25., 17.3205)  {hexagon};
\pic at (19.5, 19.9186)  {hexagon};
\pic at (14., 22.5167)  {hexagon};
\pic at (13.5, 28.5788)  {hexagon};
\pic at (19., 25.9808)  {hexagon};
\pic at (24.5, 23.3827)  {hexagon};
\pic at (18.5, 32.0429)  {hexagon};
\draw[black, ultra thick] (0, 0)-- (37,0)-- (18.5, 32.0429)-- (0, 0);
\end{tikzpicture}}}&$\gamma_{4,1}(T_n)\leq\begin{cases}6k(3k+2)&\mbox{if $\beta=0$}\\6(k+1)(3k+5)&\mbox{if $\beta\neq0$}\end{cases}$\\\hline
         $(4,2)$&$n=14k+\beta$, $0\leq \beta\leq 13$ &\resizebox{1.6in}{!}{

\tikzset{
    hexagon/.pic={
\coordinate (0;0) at (0,0); 
\foreach \c in {1,...,2}{%
\foreach \i in {0,...,5}{%
\pgfmathtruncatemacro\j{\c*\i}
\coordinate (\c;\j) at (60*\i:\c);  
} }

%
%
 
\fill [violet] (0;0) circle (4.5pt);
\newdimen\R
\R=3 cm
\filldraw[fill=violet!80!violet,opacity=0.15] (0:\R) \foreach \x in {60,120,...,359} {
                -- (\x:\R)
            }-- cycle (90:\R);
\draw[thick, violet] (0:\R) \foreach \x in {60,120,...,359} {
                -- (\x:\R)
            }-- cycle (90:\R);
            }}
\resizebox{1\textwidth}{!}{%
 \begin{tikzpicture}
    \foreach \row in {0, 1,...,14} {
        \draw[gray, thick] ($\row*(0.5, {0.5*sqrt(3)})$) -- ($(14,0)+\row*(-0.5, {0.5*sqrt(3)})$);
        \draw[gray, thick] ($\row*(1, 0)$) -- ($(14/2,{14/2*sqrt(3)})+\row*(0.5,{-0.5*sqrt(3)})$);
        \draw[gray, thick] ($\row*(1, 0)$) -- ($(0,0)+\row*(0.5,{0.5*sqrt(3)})$);
    }
\pic at (0,0)  {hexagon};
\pic at (14,0)  {hexagon};
\pic at (7., 12.1244)  {hexagon};
\pic at (5., -1.73205)  {hexagon};
\pic at (4., 3.4641)  {hexagon};
\pic at (9., 1.73205)  {hexagon};
\pic at (3., 8.66025)  {hexagon};
\pic at (8., 6.9282)  {hexagon};
\pic at (13., 5.19615)  {hexagon};
\draw[black, ultra thick] (0, 0)-- (14,0)-- (7., 12.1244)-- (0, 0);
\end{tikzpicture}}
}&$\gamma_{4,2}(T_n)\leq\begin{cases}3k(k+2)&\mbox{if $\beta=0$}\\3(k+1)(k+3)&\mbox{if $\beta\neq0$}\end{cases}$\\\hline
          $(4,3)$&$n=21k+\beta$, $0\leq \beta\leq 20$ &\resizebox{1.6in}{!}{
\tikzset{
    hexagon/.pic={
\coordinate (0;0) at (0,0); 
\foreach \c in {1,...,2}{%
\foreach \i in {0,...,5}{%
\pgfmathtruncatemacro\j{\c*\i}
\coordinate (\c;\j) at (60*\i:\c);  
} }

%
%
 
\fill [violet] (0;0) circle (4.5pt);
\newdimen\R
\R=3 cm
\filldraw[fill=violet!80!violet,opacity=0.15] (0:\R) \foreach \x in {60,120,...,359} {
                -- (\x:\R)
            }-- cycle (90:\R);
\draw[thick, violet] (0:\R) \foreach \x in {60,120,...,359} {
                -- (\x:\R)
            }-- cycle (90:\R);
            }}
\resizebox{1\textwidth}{!}{%
 \begin{tikzpicture}
    \foreach \row in {0, 1,...,21} {
        \draw[gray, thick] ($\row*(0.5, {0.5*sqrt(3)})$) -- ($(21,0)+\row*(-0.5, {0.5*sqrt(3)})$);
        \draw[gray, thick] ($\row*(1, 0)$) -- ($(21/2,{21/2*sqrt(3)})+\row*(0.5,{-0.5*sqrt(3)})$);
        \draw[gray, thick] ($\row*(1, 0)$) -- ($(0,0)+\row*(0.5,{0.5*sqrt(3)})$);
    }
\pic at (0,0)  {hexagon};
\pic at (21,0)  {hexagon};
\pic at (10.5, 18.1865)  {hexagon};
\pic at (4.5, -0.866025)  {hexagon};
\pic at (9., -1.73205)  {hexagon};
\pic at (13.5, -2.59808)  {hexagon};
\pic at (3., 3.4641)  {hexagon};
\pic at (7.5, 2.59808)  {hexagon};
\pic at (12., 1.73205)  {hexagon};
\pic at (16.5, 0.866025)  {hexagon};
\pic at (1.5, 7.79423)  {hexagon};
\pic at (6., 6.9282)  {hexagon};
\pic at (10.5, 6.06218)  {hexagon};
\pic at (15., 5.19615)  {hexagon};
\pic at (19.5, 4.33013)  {hexagon};
\pic at (4.5, 11.2583)  {hexagon};
\pic at (9., 10.3923)  {hexagon};
\pic at (13.5, 9.52628)  {hexagon};
\pic at (18., 8.66025)  {hexagon};
\pic at (7.5, 14.7224)  {hexagon};
\pic at (12., 13.8564)  {hexagon};
\pic at (16.5, 12.9904)  {hexagon};
\draw[black, ultra thick] (0, 0)-- (21,0)-- (10.5, 18.1865)-- (0, 0);
\end{tikzpicture}}
}&$\gamma_{4,3}(T_n)\leq\begin{cases}2k(5k+6)&\mbox{if $\beta=0$}\\2(k+1)(5k+11)&\mbox{if $\beta\neq0$}\end{cases}$\\\hline
        $(t,t)$&$n=tk+\beta$, $0\leq \beta\leq t-1$ &\resizebox{1.6in}{!}{
\tikzset{
    hexagon/.pic={
\coordinate (0;0) at (0,0); 
\foreach \c in {1,...,2}{%
\foreach \i in {0,...,5}{%
\pgfmathtruncatemacro\j{\c*\i}
\coordinate (\c;\j) at (60*\i:\c);  
} }

%
%
 
\fill [violet] (0;0) circle (7pt);
\newdimen\R
\R=20 cm
\filldraw[fill=violet!80!violet,opacity=0.15] (0:\R) \foreach \x in {60,120,...,359} {
                -- (\x:\R)
            }-- cycle (90:\R);
\draw[thick, violet] (0:\R) \foreach \x in {60,120,...,359} {
                -- (\x:\R)
            }-- cycle (90:\R);
            }}
\resizebox{1\textwidth}{!}{%
 \begin{tikzpicture}
    \foreach \row in {0, 1,...,21} {
        \draw[gray, thick] ($\row*(0.5, {0.5*sqrt(3)})$) -- ($(21,0)+\row*(-0.5, {0.5*sqrt(3)})$);
        \draw[gray, thick] ($\row*(1, 0)$) -- ($(21/2,{21/2*sqrt(3)})+\row*(0.5,{-0.5*sqrt(3)})$);
        \draw[gray, thick] ($\row*(1, 0)$) -- ($(0,0)+\row*(0.5,{0.5*sqrt(3)})$);
    }
\pic at (0,0)  {hexagon};
\pic at (21,0)  {hexagon};
\pic at (10.5, 18.1865)  {hexagon};
\draw[black, ultra thick] (0, 0)-- (21,0)-- (10.5, 18.1865)-- (0, 0);
\end{tikzpicture}}
}&$\gamma_{t,t}(T_n)\leq\begin{cases}3k&\mbox{if $\beta=0$}\\3(k+1)&\mbox{if $\beta\neq0$}\end{cases}$\\\hline
    \caption{Table of upper bounds for $\gamma_{t,r}(T_n)$ when $(t,r)\in\{(2,1), (3,1), (3,2),(4,1),(4,2),(4,3),(t,t)\}$.}
    \label{tableofbounds2}
    \end{longtable}

\section{Open Problems}
In Section \ref{sec:newbounds} we provided a way to use the $(t,r)$ broadcast dominating patterns in Theorem \ref{theorem:infinitedomination} to find upper bounds for the $(t,r)$ broadcast domination numbers for all finite matchstick graphs $T_n$. Using this technique we established bounds for $\gamma_{t,r}(T_n)$, for $(t,r)\in\{(2,1),(3,1),(3,2),(4,1),(4,2),(4,3),(t,t)\}$. However, we never accounted for a rearrangement of the boundary dominators so as to be more efficient in the domination of the boundary of $T_n$. This leads to the first open problem.

\begin{question}\label{q:1}
Accounting for a rearrangement of the dominators on the boundary of $T_n$, give better bounds for $\gamma_{t,r}(T_n)$ for $(t,r)\in\{(2,1),(3,1),(3,2),(4,1),(4,2),(4,3),(t,t)\}$.
More generally,
by using the technique presented in Section \ref{sec:newbounds} and by reducing the number of dominators needed on the boundary of $T_n$, find sharper bounds for $\gamma_{t,r}(T_n)$, for $t>r\geq 1$.
\end{question}
Another problem of interest would be to give lower bounds akin to the intuitive bound provided for $\gamma_{t,1}(T_n)$ in Lemma \ref{lemma:lowerbound}.

Lastly, in Section \ref{sec:infinitegrid} we presented two efficient $\tr$ broadcast dominating patterns for $\T$, one being a mirror image of the other. One can then ask:
\begin{question}
Are there other efficient $\tr$ broadcast dominating patterns for $\T$?
\end{question}

\section*{Acknowledgments}
We thank Erik Insko for introducing us to $(t,r)$ broadcast domination and for his many helpful conversations during the completion of this manuscript.

\bibliography{triangles}

\end{document}